\documentclass[11pt,reqno]{amsart}
\usepackage{amsmath}
\usepackage{amsthm}
\usepackage{graphicx}
\usepackage{amsfonts}
\usepackage{color}
\usepackage{amssymb}
\usepackage{enumerate}
\usepackage{soul}
\usepackage[margin=1.1in]{geometry}

\numberwithin{equation}{section}

\newtheorem{theorem}{Theorem}[section]
\newtheorem{lemma}[theorem]{Lemma}
\newtheorem{proposition}[theorem]{Proposition}
\newtheorem{corollary}[theorem]{Corollary}

\theoremstyle{definition}

\newtheorem{definition}[theorem]{Definition}
\newtheorem{remark}[theorem]{Remark}

\newtheorem*{standingassumption*}{Standing assumptions}

\def\E{{\mathbb E}}
\def\R{{\mathbb R}}
\def\N{{\mathbb N}}
\def\NN{{\mathcal N}}
\def\PP{{\mathbb P}}

\def\P{{\mathcal P}}

\def\M{{\mathcal M}}

\def\X{{\mathcal X}}
\def\Y{{\mathcal Y}}

\def\D{{\mathcal D}}
\def\W{{\mathcal W}}

\def\A{{\mathcal A}}
\def\F{{\mathcal F}}
\def\C{{\mathcal C}}
\def\S{{\mathcal S}}
\def\Y{{\mathcal Y}}
\def\Z{{\mathcal Z}}
\def\CY{{\mathfrak C}(\Y)}

\def\DG{{\mathcal D}}

\def\const{R}

\title[Rare equilibria in large games]{Rare Nash Equilibria and the Price of Anarchy in Large Static Games}

\author{Daniel Lacker}
\thanks{The first author was supported by the National Science Foundation under Award No. DMS 1502980.}
\author{Kavita Ramanan}
\thanks{The second author was partially supported by the grant NSF CMMI 1538706.}

\begin{document}

\begin{abstract}
We study a static game played by a finite number of agents, in which agents are assigned independent and identically distributed random types and each agent minimizes its objective function by choosing from a set of admissible actions that depends on its type. The game is anonymous in the sense that the objective function of each agent depends on the actions of other agents only through the empirical distribution of their type-action pairs. We study the asymptotic behavior of Nash equilibria, as the number of agents tends to infinity, first by deriving laws of large numbers characterizes almost sure limit points of Nash equilibria in terms of so-called Cournot-Nash equilibria of an associated nonatomic game. Our main results are large deviation principles that characterize the probability of rare Nash equilibria and associated conditional limit theorems describing the behavior of equilibria conditioned on a rare event. The results cover situations when neither the finite-player game nor the associated nonatomic game has a unique equilibrium. In addition, we study the asymptotic behavior of the price of anarchy, complementing existing worst-case bounds with new probabilistic bounds in the context of congestion games, which are used to model traffic routing in networks.
\end{abstract}



\maketitle

\section{Introduction}

\subsection{Model Description}
\label{subs-intro-model}

We consider a static game played by $n$ agents $i = 1, \ldots, n$, in
which  the $i^{\text{th}}$ agent is assigned a type $w_i$ from a 
type space $\W$, and is allowed to choose an action $x_i$ from a
subset ${\mathcal C} (w_i)$ of the action space $\X$, 
so as to minimize 
its objective function $J_i^n (w_1, \ldots, w_n, x_1, \ldots, x_n)$,
which can be viewed as a cost. 
Both $\W$ and $\X$ are assumed to be metric spaces, and 
for $w \in \W$, $\C(w)$ represents the set of admissible actions
allowed for an agent of type $w$. 
We restrict our attention to  anonymous games,   
in which  an individual agent's objective function is influenced by its own
type and action,  but depends on  the other agents' types and actions
only through the empirical distribution of type-action pairs (rather
than on the full configuration of types and actions of individual agents), 
and in addition, the form of this dependence is the  same for all agents. 
More precisely, if we let $\delta_{(w, x)}$ denote the Dirac delta
measure at the point $(w, x) \in \W \times \X$, 
 the objective function of the $i^{\text{th}}$ agent in the $n$-player game takes the form 
\[ J_i^n (w_1, \ldots, w_n, x_1, \ldots, x_n) = F \left(\frac{1}{n} \sum_{k=1}^n \delta_{(w_k, x_k)}, w_i, x_i \right), 
\]
for a suitable function $F: \P (\W \times \X) \times \W \times \X \mapsto \R$, where 
$\P(\W\times\X)$ is the set of Borel probability measures on
$\W\times\X$.  
We are interested in  properties of Nash equilibria when the number of players is large
and seek to understand  the behavior of agents in terms of their
\emph{types}, not their \emph{names}.    
Here,  a \emph{Nash equilibrium with type vector} $\vec{w} =
(w_1,\ldots,w_n) \in \W^n$ is any vector $(x_1,\ldots,x_n) \in \X^n$
such that for each $i$,  $x_i$ lies in the set  $\C(w_i)$  of
admissible actions and the objective function satisfies 
\[
J^n_i(w_1,\ldots,w_n,x_1,\ldots,x_n) = \inf_{y \in \C(w_i)}J^n_i(w_1,\ldots,w_n,x_1,\ldots,x_{i-1},y,x_{i+1},\ldots,x_n).
\]
Note that an agent's type $w_i$ influences not only the 
objective function but also the set $\C(w_i)$ of admissible actions. 
For simplicity, we work exclusively with pure strategies, although we
refer the interested reader to  Blanchet-Carlier \cite[Section
4]{blanchet2014nash} for extensions of a similar model setup to cover mixed strategies. 
  For pure strategies, even existence of a Nash
equilibrium in an $n$-player game is not 
always guaranteed, but we will  be concerned with the large
class of games for which Nash equilibria are known to exist (see
Section \ref{se:congestion}).  However, such games often admit multiple
equilibria, and so we will in general not assume  uniqueness of Nash
equilibria.

\subsection{Discussion of Results and Related Work.}

It is in general hard to explicitly identify the set of Nash
equilibria, especially in large games.  
Thus, we instead study 
the behavior of Nash equilibria in the limit as the number of
agents goes to infinity.
Specifically, under the assumption that the types of agents in the $n$-player game are sampled independently from a 
common type distribution $\lambda_0 \in {\mathcal P}(\W)$,  where
${\mathcal P}(\W)$ denotes the space of probability measures on $\W$,
the goal of this work is to study the asymptotic behavior of
corresponding Nash equilibria as the number of agents goes to
infinity.   
To emphasize that they are random, we will use capital letters
$\{W_i\}$ to denote the sequence of i.i.d.\ types and the
 array $\{X^n = (X_1^n,\ldots, X_n^n)\}$ of associated agent actions in the sequence of $n$-player games. 
 Under fairly general conditions (see the standing assumption in
 Section \ref{sec-mainres} below),  we first state  a strong law of
 large numbers (Theorem \ref{th:intro-limit})  that   shows that almost sure limit points of  sequences of (random) empirical
type-action distributions
$\frac{1}{n}\sum_{i=1}^n\delta_{(W_i,X_i^n)}$ associated with Nash
equilibria can be characterized as Cournot-Nash equilibria of a
certain  nonatomic game associated with the type distribution $\lambda_0$.     When there is a unique Cournot-Nash
equilibrium for the nonatomic game,   this implies almost-sure covergence, as $n \rightarrow
\infty$, of the empirical type-action distributions of Nash equilibria of $n$-player
games to  the corresponding Cournot-Nash equilibrium.

Our precise framework is  related
to several existing models in the literature. 
In particular,  the nonatomic game 
  is  similar to the model   considered by  Blanchet and Carlier
\cite{blanchet2014nash}, which is itself  a  reparametrization of the
seminal framework of Mas-Colell \cite{mas1984theorem}.  The particular
definition that we use  (see Section \ref{subs-nonatomic}) 
is a slight generalization that has two new features.   
First, it allows for the incorporation of a
constraint map $\C$ that specifies the admissible set of actions
associated with each  agent type.  This extension is necessary to cover
interesting examples such as the class of congestion games described
in Section \ref{subs-intro-mot}.  
Theorem \ref{th:intro-limit} is one of many related (and largely
equivalent) laws of large numbers in the literature on large games, 
notably
\cite{green1984continuum,housman1988infinite,carmona2004nash,kalai2004large,blanchet2014nash}.
Second, our
model involves \emph{unknown} or \emph{random} types, whereas all of
these papers work with \emph{known} or \emph{deterministic} sequences
of type vectors $(w^n_1,\ldots,w^n_n)$ satisfying
$\frac{1}{n}\sum_{i=1}^n\delta_{w^n_i} \rightarrow \lambda_0$.   Although
limited to complete information and homogeneous beliefs, our model
setup is nonetheless also reminiscent of Harsanyi's formalism of
\emph{Bayesian games} \cite{harsanyi1967games}.

 The primary focus of this work is the estimation of the probability 
that a Nash equilibrium of an $n$-player game  makes    a large
deviation from the law of large numbers limit when $n$  is large.   
Our first main set of results, 
 stated in Section
\ref{subs-ld1},   concern the large deviations
behavior of any sequence of
(random) empirical type-action distributions associated with Nash
equilibria, under the assumption that there is a  unique Cournot-Nash equilibrium for the corresponding
  nonatomic game. 
Specifically,   Theorem \ref{th:intro-LDP} establishes a 
large deviations principle  (LDP) 
 for such a sequence, which provides precise  asymptotic estimates of the
 exponential rate of decay of probabilities of the occurrence of rare
 Nash  equilibria (i.e., those  that are far from the Cournot-Nash
 equilibrium), and the exponential decay rate  is expressed in terms of
 quantities that are derived from the more tractable nonatomic game. 
Establishing an LDP (as opposed to just obtaining large deviation
bounds) sheds light on the behavior of Nash 
equilibria conditioned on a rare event, as exemplified by  the 
conditional limit result  in  Theorem
\ref{th:intro-conditional-limit}.

Uniqueness of the Cournot-Nash equilibrium holds for many 
 nonatomic games,  including the important class of potential games 
with strictly convex potential.
This covers many congestion games, discussed in more detail in Section \ref{se:congestion}. The foundational work on finite potential games is \cite{monderer-shapley}, but we refer to \cite{blanchet2015optimal} for interesting developments on potential games for general (possibly uncountable) type and action spaces.  In more recent work
\cite{blanchet2014remarks,blanchet2016computation}, Blanchet et
al.\ exploit a connection with optimal transport to develop methods for
computing Cournot-Nash  equilibria even for non-potential games.

However, there are also cases of interest for which the nonatomic game admits multiple
equilibria.   To address this situation,   in Section \ref{se:intro:LDP-set} we also consider the large deviation behavior of the 
 set  of  empirical distributions induced by all the 
 Nash equilibria of an $n$-player game.  We first state an analogous 
law of large numbers result in Theorem \ref{th:intro-limit-setvalued} that
shows convergence of the sequence of sets of Nash equilibria to the corresponding set of Cournot-Nash equilibria for the
nonatomic game, and then establish a corresponding LDP in Theorem 
\ref{th:LDP-setvalued}.   The choice of topology on the space of sets
of distributions for this LDP is rather subtle.   One needs to 
identify a topology that is weak enough for the LDP to hold, but
strong enough that the LDP can provide useful information. 
We show  in Corollary \ref{co:setvalued} that, indeed, our LDP
provides  interesting information on the probability of outliers or
rare equilibria even in the non-unique setting.  
Additionally, as elaborated below, in Section \ref{se:PoA}, we also 
show that the LDP is useful for obtaining  interesting asymptotic
results on the price of anarchy.

Our results  appear to be the first  LDPs for any kind for large games.
Philosophically, the paper that is closest to ours is that of Menzel
\cite{menzel2016inference}, which adopts a similar statistical
perspective to large-$n$ asymptotics in  order to derive a central
limit theorem in addition to a law of large numbers like Theorem
\ref{th:limit}. Although the model specification in \cite{menzel2016inference} is very different from our own, Menzel interprets his results as
``expansions'' of the $n$-player games around the  nonatomic game
``limit'', which is useful because the latter is typically more
tractable. Likewise, our results provide asymptotics for $n$-player
quantities in terms of quantities derived from the associated
nonatomic game, namely, the rate functions in Theorems
\ref{th:intro-LDP} and \ref{th:intro-LDP-setvalued}). 
However, rather than addressing econometric questions as in Menzel, we
focus on the probabilistic nature of equilibria
arising from a large number of random types.

Finally, we apply our large deviation analysis to derive
high-probability bounds of the so-called \emph{price of anarchy} as
the number of agents grows. 
The \emph{price of anarchy}, a term first introduced by Koutsoupias and Papadimitriou
\cite{koutsoupias1999worst}, is a measure of the degradation of efficiency 
in a system due to the selfish behavior of its agents, and it is defined roughly as follows. 
Given a type vector $\vec{w} \in \W^n$, the socially
optimal cost is the least  average cost of all players over
all associated admissible type-action pairs, and  
the price of anarchy of the $n$-player game 
is the ratio of the worst-case (or highest) average cost
induced by any  Nash equilibrium  to the corresponding socially optimum
cost (see Section \ref{se:PoA} for  a precise definition). 
The price of anarchy measures the degradation of efficiency 
in a system due to the selfish behavior of its agents.

\subsection{Motivation}
\label{subs-intro-mot}

The motivation for this study is twofold.  Firstly, 
our results apply to a class of games introduced  
by Rosenthal in \cite{rosenthal1973class} called \emph{congestion games}, 
which have found widespread application in modeling traffic routing,
in both physical and communications networks,  particularly in the field of algorithmic game theory
\cite{algorithmicgametheory}.  
In the context of traffic modeling, the congestion game is played on a network, represented by a finite graph, 
and the type of an agent is associated with a certain source-destination pair, represented by a pair of vertices in the graph.  
The distribution of types could be assumed to  be known from historical data. 
 Given a realization of
these types, the agents of the game can be viewed as drivers who competitively choose their routes (between their associated source and destination) to minimize travel time, leading to a corresponding traffic outcome  determined by the 
Nash equilibrium (or the set of Nash equilibria, when multiple exist).
  
In managing network traffic, an important quantity is the average
travel time (latency) faced by the agents.  
A central planner managing the network might prefer to assign to each
agent a \emph{socially optimal} route, which minimizes the
average travel time, but this is rarely feasible. 
When agents choose routes
selfishly, to minimize their own travel times, the resulting
social cost or average travel time 
is typically socially suboptimal,
and the price of anarchy is a popular measure of this suboptimality \cite{roughgarden2002bad}. 
In Section \ref{se:congestion}, we describe the class of congestion
games, and illustrate how our main results can be used to provide 
new probabilistic bounds on the price of anarchy for such games. 
In particular, Corollary \ref{co:PoA-congestion} shows how to
translate a bound on the price of anarchy in a nonatomic game into a 
high probability bound on the price of anarchy in the corresponding finite (but
large) game.  In particular,  our results complement existing
worst-case bounds such as those  of Christodoulou and Koutsoupias \cite{christodoulou2005price} 
on the price of anarchy for  $n$-player congestion games determined by a 
class of linear cost functions  by providing with-high-probability bounds for the
price of anarchy arising from  a fixed
 cost function in that class.

While we focus on congestion games as a motivating example, our
framework  encompasses many different types of large static games
appearing in applications, with notable examples including \emph{entry
  games}
\cite{berry1992estimation,bresnahan1991empirical,bajari2010identification}
and \emph{auctions} \cite{krishna2009auction,klemperer-guide}. In both
of these examples, it is more natural to interpret agents as
\emph{maximizing a payoff},  which  we identify with $-F$, the
  negative of the cost function.  A
prototypical entry game, borrowed from \cite{berry1992estimation}, has
$\X=\{0,1\}$, an arbitrary type space $\W$, and payoff $-F(m,w,x) =
x[f(m^x\{1\}) + g(w)]$, for some functions $f$ and $g$, where $f$ is decreasing. The action
$x=1$ means the agent ``enters the market.'' An agent that does not enter receives no payoff, while an
agent that enters receives a payoff which is decreasing in the
fraction $m^x\{1\}$ of agents that enter. All of our main results
apply to entry games, as long as $f$ and $g$ are continuous.  We
discuss entry games in somewhat more detail in Section
\ref{se:conditional-limit}, as an illustration of our conditional
limit theorem. 

On the other hand, our results do not apply to many models of
auctions, for which the payoff function is
discontinuous. More specifically, a
typical auction model has $\W = \X \subset [0,\infty)$ and a payoff
function $F(m,w,x)$ with discontinuities at points where $x =
\max\mathrm{supp}(m^x)$, where $\mathrm{supp}(m)$ represents the
support of the distribution $m$.  For instance, in an auction of a single unit
of a single good, the classical first-price auction has payoff
$-F(m,w,x) = (w-x)1_{\{x \ge \max\mathrm{supp}(m^x)\}}$,  with the
type $w$ representing the intrinsic value of the good. That is, when
the bid $x$ of a given agent becomes the maximum bid, the payoff of
the agent jumps from zero to $w-x$. It is not clear if our main
results should still hold in the presence of such discontinuities.  
Extending our results to include these other applications would be an
interesting problem for future work.

An additional motivation for our work relates to the study of Nash equilibria in
\emph{dynamic} $n$-player games, on which a vibrant literature has emerged
recently. These games 
arise in a variety of settings and are harder to analyze than
static games. 
Various law of large numbers type limit theorems and approximation
results are now fairly well understood, both in discrete time
\cite{weintraub2008markov,adlakha2008oblivious,adlakha2013mean,gomes2010discrete}
and in continuous time
\cite{lasrylionsmfg,cardaliaguet2015master,lacker2014general,fischer-mfgconnection,carmonadelarue-mfg},
and are expressed in terms of associated dynamic games with a continuum
of agents which largely go by the name of \emph{mean field games}.  
The present work grew in part out of early efforts to understand large
deviations in dynamic mean field games, especially in the case when 
the mean field game admits multiple equilibria. 
In many dynamic models, the random variables $\{W_i\}$ which we called
\emph{types} are better interpreted as \emph{noises}.  For instance,
the continuous time models typically involve controlled diffusion
processes driven by a sequence $\{W_i\}$ of  i.i.d.\ Brownian motions.  
This \emph{noise} interpretation is equally valid for the static games
of this paper, if we think of $W_i$ as a random shock
 to agent $i$.  We hope that our large deviation analysis of static
 games be useful not only on its own merit but also as a first step
 toward understanding large deviations in dynamic games.

\section{Statements of Main Results}
\label{sec-mainres}

In this section, we precisely state our results,
  the proofs of which are presented in Section \ref{se:proofs}, with some
auxiliary results required for the proofs deferred to Appendices  \ref{ap:vietoris}
 and \ref{ap:congestiongames}.
In what follows, given a metric space $\S $,  we let 
$\P(\S)$ denote the space of Borel probability measures on $\S$,
equipped  with the topology of  weak convergence.
We will refer to convergence in this topology as convergence in distribution, and denote this convergence by 
$m_n \rightarrow m$, 
which we recall   means
that $\int\varphi\,dm_n \rightarrow \int\varphi\,dm$ for every bounded
continuous function $\varphi$ on $\S$. We will most often consider
the case $\S = \W$ or $\S = \W \times \X$. 
 Throughout the paper, we make the following assumptions on the model.

\begin{standingassumption*}
\label{as-main}
The following model parameters are given: 
\begin{enumerate}
\item The \emph{action space} $\X$ is a compact metric space. 
\item The \emph{type space} is a complete separable metric space $\W$. 
\item  The constraint map $\C$, which maps elements of $\W$ to nonempty closed subsets of $\X$, 
is continuous. Here, continuity of the set-valued map $\C$ means 
 both that the graph $\mathrm{Gr}(\C) = \{(w,x) \in \W \times \X : x \in \C(w)\}$ is closed and that, if $w_n \rightarrow w$ in $\W$ and $x \in \C(w)$, then there exist $n_k$ and $x_{n_k} \in \C(w_{n_k})$ such that $x_{n_k} \rightarrow x$. 
\item  The \emph{objective function} $F : \P(\W\times\X) \times \W \times \X \rightarrow \R$ is bounded and continuous,  
where $\P(\W\times\X) \times \W \times \X$ is equipped with the product topology. 
\end{enumerate}
\end{standingassumption*}

The compactness of $\X$ assumed in (1) is important but could likely be replaced by a coercivity assumption on $F$. 
In our main application to congestion games, both sets $\W$ and $\X$ are finite, in which case the continuity assumptions (3) and (4) hold automatically. 
Fix throughout the paper an arbitrary probability space
$(\Omega,\F,\PP)$, and assume it is rich enough to support all of the
random variables of interest. 
We also assume throughout that  for each $n \ge 1$ and each type vector $\vec{w} =
(w_1,\ldots,w_n) \in \W^n$, the  
set of Nash equilibria with type vector $\vec{w}$ is non-empty.

\subsection{Nonatomic games and Cournot-Nash equilibria}
\label{subs-nonatomic}

Let $\hat{x}^n_i: \W^n \mapsto \X$ be measurable functions such that $(\hat{x}^n_1(\vec{w}),\ldots,\hat{x}^n_n(\vec{w}))$ is a Nash equilibrium with type vector $\vec{w}$, for each $\vec{w} \in \W^n$ (it is
shown in Lemma \ref{le:measurable-selection} that such a measurable
selection always exists under our assumptions).   Now,  suppose that
$W_1,\ldots,W_n$ are the i.i.d.\ types  sampled from a 
distribution $\lambda_0 \in {\mathcal P}(\W)$, which we fix once and for all. Let
$X^n_i=\hat{x}^n_i(W_1,\ldots,W_n)$ denote the associated random Nash equilibrium vector. 
The equilibrium type-action distribution is the random probability measure (on $\W\times\X$) given by 
\[
\mu_n := \frac{1}{n}\sum_{i=1}^n\delta_{(W_i,X^n_i)}.
\]
Our main results concern the asymptotic behavior of $\{\mu_n\}$,
which, as mentioned in the introduction, is 
expressed in terms of  equilibria for the corresponding
\emph{nonatomic game}, also called \emph{Cournot-Nash equilibria},
 defined as follows. 

\begin{definition}[Cournot-Nash equilibria]
\label{def-CNeqb}
For $\lambda \in \P(\W)$,  the set $\M(\lambda)$ of \emph{Cournot-Nash
  equilibria with type distribution $\lambda$} is defined as the set of $m \in
\P(\W\times\X)$ with first marginal equal to $\lambda$ that satisfy 
\[
m\left\{(w,x) \in \W\times\X : x \in \C(w), \ F(m,w,x) = \inf_{y \in \C(w)}F(m,w,y)\right\}=1, 
\]
that is,  $x \in \C(w)$ and $F(m,w,x) = \inf_{y \in \C(w)}F(m,w,y)$ hold for $m$-almost every $(w,x)$.
\end{definition}

Intuitively, a Cournot-Nash equilibrium $m \in \M(\lambda)$ describes
an equilibrium distribution of type-action pairs in a game consisting of a
continuum of infinitesimally small agents.  Although, in an $n$-player game (pure-strategy) Nash equilibria need not
in general exist,
 a standard argument in Proposition \ref{pr:existence}
below (adapted from \cite{mas1984theorem}) shows that there always
exists a  Cournot-Nash equilibrium, i.e., $\M({\lambda}) \neq
\emptyset$ for all $\lambda \in \P(\W)$.

\subsection{Large deviation results for sequences of Nash equilibria}
\label{subs-ld1}

In a Cournot-Nash equilibrium 
 no individual agent has direct influence on the equilibrium
 distribution $m$. Agents thus optimize independently, facing i.i.d.\
 types, and a law of large numbers heuristic suggests that $m$ should,
 in equilibrium, agree with the distribution of type-action
 pairs. This heuristic is justified by the following rigorous result: 

\begin{theorem} \label{th:intro-limit} 
Given that agent types are i.i.d.\ with  distribution $\lambda_0 \in \P (\W)$, 
for any metric $d$  on $\P(\W\times\X)$ compatible with weak
convergence, it holds with probability one that
$d(\mu_n,\M ({\lambda_0})) := \sup_{m \in \M ({\lambda_0})}d(\mu_n,m) \rightarrow 0$.
\end{theorem}

We prove a somewhat more general form of this result in Section  \ref{se:proofs} (see  
Theorem \ref{th:limit}(ii) therein), which allows for  approximate Nash equilibria 
and correlated types, although we do not push this result to the utmost
generality because it is not the main novelty of the paper.

\begin{remark}
\label{rem-mixed}
For an idea of how to adapt Theorem \ref{th:intro-limit} to mixed
strategies, which we do not explore in this paper, see \cite[Theorem
4.2]{blanchet2014nash}. 
\end{remark}

We know from Theorem \ref{th:intro-limit} that the limit points of
$\{\mu_n\}$ lie in the set $\M({\lambda_0})$.  Our first main result, Theorem \ref{th:intro-LDP} below, lets us estimate how unlikely it is that $\mu_n$ remains ``far'' in some sense from this limiting set.  
 To state the theorem precisely, we introduce some definitions. 
Write $\lambda \ll \lambda_0$ when $\lambda$ is absolutely continuous with respect to $\lambda_0$, and define the relative entropy as usual by
\begin{align}
H(\lambda | \lambda_0) := \int_\W\frac{d\lambda}{d\lambda_0}\log\frac{d\lambda}{d\lambda_0}\,d\lambda, \text{ for } \lambda \ll \lambda_0, \quad\quad H(\lambda | \lambda_0) = \infty \text{ otherwise}. \label{def:entropy}
\end{align}
Define 
\begin{equation}
\label{def-cneqb}
\M = \bigcup_{\lambda \in \P(\W)}\M(\lambda), 
\end{equation}
 to be the set of
all Cournot-Nash equilibria, with any type distribution. For $m \in
\P(\W\times\X)$ let $m^w$ and $m^x$ denote the first and second
marginals, respectively, of $m$.   Throughout the paper, we adopt the
convention that $\inf\emptyset = \infty$  and $\sup \emptyset =
  -\infty$.

\begin{theorem} \label{th:intro-LDP}
Assume that $\M({\lambda})$ is a singleton for each
$\lambda \in \P(\W)$ with $\lambda \ll \lambda_0$. Then, for every
measurable set $A \subset \P(\W\times\X)$, 
\begin{align*}
\limsup_{n\rightarrow\infty}\frac{1}{n}\log\PP(\mu_n \in A) &\le -\inf_{m \in \overline{A} \cap \M}H(m^w|\lambda_0), \\
\liminf_{n\rightarrow\infty}\frac{1}{n}\log\PP(\mu_n \in A) &\ge -\inf_{m \in A^\circ  \cap \M}H(m^w|\lambda_0),
\end{align*}
where $A^\circ $ and $\overline{A}$ denote the interior and closure,
respectively, of $A$. In other words, $\{\mu_n\}$ satisfies  an LDP  
 on $\P(\W\times\X)$ with (good) rate function 
\[  m \mapsto    \begin{cases}
H(m^w | \lambda_0) &\text{if } m \in \M, \\
\infty &\text{otherwise.}
\end{cases}
\]
\end{theorem}

Theorem \ref{th:intro-LDP} follows from a more general 
result, Theorem \ref{th:LDP}, proved in Section \ref{subs-pf-ldp}. 
In applications, one can use Theorem \ref{th:intro-LDP} to estimate
the asymptotic probabilities of what are best interpreted as
\emph{rare equilibrium outcomes}.  
Given an event $A \subset \P(\W\times\X)$ whose closure is disjoint
 from $\M({\lambda_0})$,  for example,  $A = \{m \in \P(\W\times\X) : d(m,\M({\lambda_0})) \ge
\epsilon\}$, 
Theorem \ref{th:intro-limit} says that $\PP(\mu_n \in A)
\rightarrow 0$, and Theorem \ref{th:intro-LDP} says that this happens
exponentially quickly, making  the event 
\emph{rare} in the sense that  roughly $\PP(\mu_n \in A) \approx
e^{-nc_A}$ for a constant $c_A > 0$. 
Indeed, it is easy to show (see Lemma \ref{le:inf>0} below) that $c_A := \inf_{m \in A  \cap
  \M}H(m^w|\lambda_0) > 0$ for  the particular set  $A$ chosen above, 
 so that the upper bound of Theorem \ref{th:intro-LDP} is nontrivial.

For a more tangible application, for a closed set $B \subset \X$ we can estimate the probability
\[
\PP\left(X^n_i \in B \text{ for some } i\right) = \PP\left(\mathrm{supp}(\mu_n^w) \cap B \neq \emptyset\right),
\]
that the action of some agent belongs to the set $B$;
here $\mathrm{supp}(m)$ denotes the support of a measure $m$. For
instance, in a traffic congestion game, this event could represent
some agent utilizing a seemingly inefficient or slow route. This event
is ``rare'' as long as $B$ does not intersect the support of
$m_0^x$ 
where $m_0$ is the unique element of
$\M({\lambda_0})$.  Again, by ``rare'' we mean $\inf\{H(m^w|\lambda_0) : m \in \M, \ \mathrm{supp}(m^x) \cap B \neq \emptyset\} > 0$, so that the upper bound of Theorem \ref{th:intro-LDP} is nontrivial.

Theorem \ref{th:intro-LDP} is of course related to Sanov's theorem and indeed reduces to it in degenerate cases (e.g., when $\X$ is a singleton).
Our framework also admits an analog of Cram\'er's theorem: If $\X$ is a subset of a Euclidean space, then we can estimate probabilities involving the \emph{average} of agents' actions, such as $\PP(\frac{1}{n}\sum_{i=1}^nX^n_i \in B)$ for $B \subset \X$.

A full LDP, which explicitly characterizes  asymptotic large deviation upper and lower
bounds provides information [about the system] that cannot be obtained by just
one-sided bounds.  Specifically,  in the spirit of the so-called Gibbs conditioning principle
(see, for instance, \cite{csiszar1984sanov,dembozeitouni}), 
the LDP  of Theorem \ref{th:intro-LDP} can be
used to derive the following \emph{conditional limit theorem}, which
tells us about the typical behavior of $\mu_n$ given that a rare event
of the form $\{\mu_n \in A\}$ occurs:

\begin{theorem} \label{th:intro-conditional-limit} 
Let $A \subset \P(\W\times\X)$ be measurable  and define 
\begin{eqnarray}
\label{IA}
I(A)  & :=  & \inf\left\{H(\lambda | \lambda_0) : \lambda \ll \lambda_0, \
  \overline{A} \cap \M(\lambda) \neq \emptyset\right\}, \\
S(A) & := &  \left\{m \in \overline{A} \cap \M : H(m^w |\lambda_0) = I(A)\right\}. \label{def:S(A)}
\end{eqnarray}
Suppose $I(A) < \infty$. Then  $S(A)$ is nonempty and compact.  
Assume that 
\begin{align}
I(A)   = \inf\left\{H(\lambda | \lambda_0) : \lambda \ll \lambda_0, \
  \M(\lambda) \subset A^\circ\right\},   \label{def:conditional-assumption}
\end{align}
and  also that $\mathbb{P}(\mu_n \in A)$ is nonzero for all sufficiently
large $n$.
Then, letting $d$ denote any
metric  on $\P(\W\times\X)$ compatible with weak convergence, for each
$\epsilon > 0$ there exists $c > 0$ such that, for all sufficiently
large $n$, 
\begin{align}
\label{decay}
\PP\left(\left. d(\mu_n,S(A)) \ge \epsilon \right| \mu_n \in A\right) \le e^{-cn}.
\end{align}
In particular, every limit point of the sequence of conditional distributions
of $\mu_n$ given $\{\mu_n \in A\}$, $n \in \mathbb{N}$, is supported on
the set $S(A)$. If $S(A) = \{\nu\}$ is a singleton, then these
conditional distributions converge to the point mass at $\nu$.
Finally, if $\M(\lambda)$ is a singleton for every $\lambda \ll
\lambda_0$, then in fact \eqref{def:conditional-assumption} is
equivalent to the following condition: 
\begin{align}
I(A) = \inf_{m \in \overline{A} \cap \M}H(m^w | \lambda_0)  =
 \inf_{m \in A^\circ \cap \M}H(m^w |
 \lambda_0).  \label{def:conditional-assumption-original}
\end{align}  
\end{theorem}

The proof of this conditional limit theorem is given in
Section \ref{se:conditional-limit}. 
The challenge in applying Theorem \ref{th:intro-conditional-limit}
lies in checking the assumption \eqref{def:conditional-assumption}, or
equivalently \eqref{def:conditional-assumption-original} when there is
uniqueness,  and
also showing that the set $S(A)$ of \eqref{def:S(A)} is a
singleton.  The key difficulty is that the set $\M$ is
never convex in nontrivial cases, which makes the minimization
problems in \eqref{def:conditional-assumption-original} more
difficult than those that  arise from the usual Gibbs conditioning
principle.   However, these assumptions can be verified in 
  several cases of interest.  As an illustration, in Section
  \ref{se:conditional-limit} we discuss in  detail
  a simple example of an entry game in which both  assumptions
  can be verified.

Theorem \ref{th:intro-LDP} applies to a given sequence
(more precisely, triangular array)  of Nash equilibria $\{X^n_i,1 \le i \le n\}_{n
\in \mathbb{N}}$, under a crucial uniqueness assumption. 
Notice that the uniqueness assumption is imposed \emph{only at the
  limit}, for the Cournot-Nash equilibrium, and no uniqueness is
required of the equilibria of $n$-player games. It is evident that
some kind of uniqueness assumption at the limit is necessary. 
Suppose, for instance, that $\X$ contains at least two elements, that
$\C(w)=\X$ for all $w$, and that the cost function is the trivial $F
\equiv 0$. Then there is no hope for an LDP  
 because \emph{any choice of actions} is a Nash equilibrium.
Uniqueness is known to hold in various particular models as well as
for a broad class of games known as \emph{potential games}, at least
when the potential is strictly convex, and we will encounter a class
of examples in our discussion of congestion games in Section
\ref{se:congestion}.  Nonetheless, uniqueness is not to be expected in
general. 

\subsection{Large deviation results for the set of equilibria} \label{se:intro:LDP-set}

We now address the case when there are multiple Cournot-Nash
equilibria for the limiting nonatomic game.   
Let $\widehat{\NN}_n : \W^n \rightarrow 2^{\P(\W\times\X)}$ denote the set-valued map that assigns to each type vector the corresponding set of equilibrium type-action distributions:
\begin{align}
\widehat{\NN}_n(w_1,\ldots,w_n) := \left\{\frac{1}{n}\sum_{i=1}^n\delta_{(w_i,x_i)} : (x_1,\ldots,x_n) \text{ is Nash for types } (w_1,\ldots,w_n)\right\}. \label{def:intro-Nn}
\end{align}
Again, let $\{W_i\}$ be a sequence of i.i.d.\ $\W$-valued random
variables with distribution $\lambda_0$, and let $\NN_n = \widehat{\NN}_n(W_1,\ldots,W_n)$ denote the random set of equilibrium type-action distributions.

It is shown in Proposition \ref{pr:Nash-UHC} that $\widehat{\NN}_n(\vec{w})$ and $\M(\lambda)$ are always closed sets. 
Thus, in the following theorem, we topologize the space $\mathfrak{C}$
of closed subsets of $\P(\W\times\X)$ with the \emph{upper Vietoris
  topology}, generated by the base of open sets of the form $\{A \in
\mathfrak{C} : A \subset E\}$, where $E$ is an open 
 subset of   $\P(\W\times\X)$. 
See Appendix \ref{ap:vietoris} for a short discussion of the basic
properties of this topology, the most important of which is that it
topologizes \emph{upper hemicontinuity} of set-valued maps.   
First, Theorem \ref{th:intro-limit-setvalued} states that  $\NN_n$
converges almost surely to $\M({\lambda_0})$, thus establishing a
law-of-large numbers result in  the upper Vietoris topology, which we prove in Section \ref{subs-pf-lln1}.

\begin{theorem} \label{th:intro-limit-setvalued}
The sequence of random sets $\{\NN_n\}$ converges almost surely to
$\M({\lambda_0})$. 
\end{theorem}

The next main result is an LDP  
for the \emph{set
  of Nash equilibria}. This not only does away with {the uniqueness
assumption on Cournot-Nash equilibria imposed in Theorem \ref{th:intro-LDP}, 
 but it carries more information than
Theorem \ref{th:intro-LDP} even when there is uniqueness.  
As shown in Remark \ref{rem-LDP-setvalued}, 
Theorem \ref{th:intro-LDP-setvalued} follows from a more 
general result,  Theorem \ref{th:LDP-setvalued},  established in 
Section \ref{subs-pf-ldp}. 

\begin{theorem} \label{th:intro-LDP-setvalued}
For Borel sets $\mathfrak{U} \subset \mathfrak{C}$, 
\begin{align*}
\limsup_{n\rightarrow\infty}\frac{1}{n}\log\PP(\NN_n \in \mathfrak{U}) &\le -\inf\{H(\lambda | \lambda_0) : \lambda \in \P(\W), \ \M(\lambda) \in \overline{\mathfrak{U}}\}, \\
\liminf_{n\rightarrow\infty}\frac{1}{n}\log\PP(\NN_n \in \mathfrak{U}) &\ge -\inf\{H(\lambda | \lambda_0) : \lambda \in \P(\W), \ \M(\lambda) \in \mathfrak{U}^\circ \}.
\end{align*}
In other words, $\{\NN_n\}$ satisfies an LDP  
on $\mathfrak{C}$ with (good) rate function
\begin{align}
A \mapsto \inf\left\{H(\lambda | \lambda_0) : \lambda \in \P(\W), \ \M(\lambda) = A\right\}. \label{def:setvalued-rate-function}
\end{align}
\end{theorem}

At first, this theorem may appear too abstract to be useful, especially given that the upper Vietoris topology is rather coarse (even non-Hausdorff). On the contrary, it yields several interesting concrete results, a key example of which stems from the following simple corollary.

\begin{corollary} \label{co:setvalued}
If $E \subset \P(\W\times\X)$ is closed, then
\begin{align}
\limsup_{n\rightarrow\infty}\frac{1}{n}\log\PP\left(\NN_n \cap E \neq \emptyset\right) &\le -\inf\left\{H(m^w|\lambda_0) : m \in \M \cap E\right\}. \label{def:outliers}
\end{align}
If $E$ is open, then
\begin{align*}
\liminf_{n\rightarrow\infty}\frac{1}{n}\log\PP\left(\NN_n \subset E\right) &\ge -\inf\left\{H(\lambda|\lambda_0) : \lambda \in \P(\W), \ \M_\lambda \subset E\right\}.
\end{align*}
\end{corollary}

Corollary \ref{co:setvalued} can be interpreted in terms of
\emph{outliers}, or rare equilibria. Indeed, the left-hand side of
\eqref{def:outliers} is the probability that there exists a Nash
equilibrium for the $n$-player game that lies in the set $E$.  If 
$\M({\lambda_0})\cap E = \emptyset$, we know from Theorem
\ref{th:intro-limit} that equilibria in $E$ should be rare when $n$ is
large in the sense that $\PP(\NN_n \cap E \neq \emptyset) \rightarrow
0$. The bound \eqref{def:outliers} shows that this probability decays
exponentially and quantifies precisely the exponential 
decay rate.  The proofs of the large deviations results in  Theorem
\ref{th:intro-LDP-setvalued} can be found in Section \ref{subs-pf-ldp} (see
Theorem \ref{th:LDP-setvalued}) and hinge on the well-known
\emph{contraction principle}, once the $n$-player games and the nonatomic game are set on a common topological space (as in Section \ref{subs-coupling}).

It should also be mentioned that a map of the form $\mathfrak{C} \ni A
\mapsto G(A) := \sup_{m \in A}g(m) \in \R$ is upper semicontinuous
whenever $g$ is upper semicontinuous. If $g$ is continuous, and if it
is constant on a set $A$, then $G$ is continuous at $A$.  These
facts (proven in Lemma \ref{le:lhcVietoris})
can  be used to derive large deviation bounds for a sequence of random
variables of the form $\sup_{m \in \NN_n}g(m)$, 
which we interpret as the worst case value of $g$, in equilibrium. The following section investigates a somewhat more complex instance of this observation.

\subsection{Price of anarchy} \label{se:PoA}

We now provide a precise definition of the price of anarchy for both
$n$-player and nonatomic games.  We assume that $F \geq 0$, which is essentially without loss of
generality due to the boundedness assumption (4). 
For each $n$ and each type vector $\vec{w} = (w_1, \ldots, w_n) \in
\W^n$, define the set of \emph{all admissible}  type-action distributions by
\begin{align}
\widehat{\A}_n(w_1,\ldots,w_n) := \left\{\frac{1}{n}\sum_{k=1}^n\delta_{(w_k,x_k)} : x_i \in \C(w_i), \ i=1,\ldots,n\right\}. \label{def:An}
\end{align}
The \emph{average cost} of the game, for a fixed type-action distribution $m \in \P(\W\times\X)$, is defined by
\begin{align}
V(m) :=  \int_{\W\times\X} F(m,w,x)\,m(dw,dx). \label{def:V}
\end{align}
Finally the price of anarchy is the ratio of the worst-case Nash equilibrium cost to the socially optimal cost, or
\[
\mathrm{PoA}_n(\vec{w}) := \frac{\sup_{m \in
    \widehat{\NN}_n(\vec{w})}V(m)}{\inf_{m \in
    \widehat{\A}_n(\vec{w})}V(m)}, 
\]
where recall the definition of $\widehat{\NN}_n$ from 
\eqref{def:intro-Nn}.  
Recall that $\widehat{\NN}_n(\vec{w})$, and thus
  $\widehat{\A}_n(\vec{w})$, 
 is non-empty due to  our standing assumption on the existence of Nash equilibria for $n$-player
games.  Assume that  $V$ is strictly positive, which by continuity implies
that $V$ is bounded from below away from zero on the non-empty compact set
$\widehat{\A}_n(\vec{w})$, for each fixed $n$ and $\vec{w} \in
\W^n$. Moreover,  $V$ is bounded since $F$ is bounded by our standing
assumption (4).   Thus, the numerator above 
is also a finite positive number. Hence,  $\mathrm{PoA}_n(\vec{w})$ is well defined.

Finally, define the price of anarchy for the nonatomic game as
follows. For $\lambda \in \P(\W)$, set 
\begin{align}
\A(\lambda) :=  \left\{m \in \P(\W\times\X) : m^w = \lambda, \ m\{(w,x) : x \in \C(w)\}=1\right\}. \label{def:A_lambda}
\end{align}
This is simply the set of all admissible type-action distributions
for the nonatomic game with type distribution $\lambda$.  The price of anarchy is then
\[
\mathrm{PoA}(\lambda) := \frac{\sup_{m \in \M(\lambda)}V(m)}{\inf_{m \in \A(\lambda)}V(m)}.
\]
Under our standing assumptions, $V$ is bounded above and by Proposition
\ref{pr:existence},  $\M (\lambda) \neq \emptyset$ for each fixed $\lambda \in
\P(\W)$.
Again, if $V > 0$ pointwise then by continuity $V$ is bounded from below away from
zero on the non-empty compact set $\A(\lambda)$, for each fixed $\lambda \in
\P(\W)$,  and $\mathrm{PoA}(\lambda)$ is well defined. 
As before, let $\{W_i\}$ be i.i.d.\ $\W$-valued random variables with
distribution $\lambda_0$. See Section \ref{se:PoA-proofs} for the
proof of the following:

\begin{proposition} \label{pr:PoA}
Assume $V>0$ pointwise. It holds almost surely that
\[
\limsup_{n\rightarrow\infty}\mathrm{PoA}_n(W_1,\ldots,W_n) \le \mathrm{PoA}(\lambda_0).
\]
Moreover,\footnote{Equivalently, $\mathrm{PoA}_n(W_1,\ldots,W_n)$
  satisfies an LDP  on $(\R \cup \{-\infty\},\tau)$ with good rate function $r \mapsto \inf\left\{H(\lambda|\lambda_0) : \lambda \in \P(\W), \ \mathrm{PoA}(\lambda) =r\right\}$, where $\tau = \{[-\infty,a) : a \in \R \cup \{-\infty\}\}$ is the lower topology.}  for each $r$,
\begin{align*}
\limsup_{n\rightarrow\infty}\frac{1}{n}\log\PP(\mathrm{PoA}_n(W_1,\ldots,W_n) \ge r) 	&\le -\inf\left\{H(\lambda|\lambda_0) : \lambda \in \P(\W), \ \mathrm{PoA}(\lambda) \ge r\right\}, \\
\liminf_{n\rightarrow\infty}\frac{1}{n}\log\PP(\mathrm{PoA}_n(W_1,\ldots,W_n) < r) 	&\ge -\inf\left\{H(\lambda|\lambda_0) : \lambda \in \P(\W), \ \mathrm{PoA}(\lambda) < r\right\}.
\end{align*}
\end{proposition}

\subsection{Congestion games} \label{se:congestion}

We now introduce the class of congestion games alluded to in Section
\ref{subs-intro-mot}. 
To specify the model, we work with a finite set $\W$ of types. Given a
finite set $E$ of \emph{elements}, the action space is the set
  $\X = 2^E\backslash \{\emptyset\}$ of nonempty subsets.  The constraint map $\C$ is arbitrary for the moment.
A continuous increasing function $c_e : [0,\infty) \rightarrow
[0,\infty)$ is given for each $e \in E$, which represents the
\emph{cost} faced by an agent when using element $e$, as a function of
the current \emph{load} or \emph{congestion} on that element.  
The cost function $F$ is defined by
\begin{equation}
\label{def-costF}
F(m,w,x) := \sum_{e \in x}c_e\left(\ell_e(m)\right), \quad \text{ where } \quad \ell_e(m) := m\{(w,x) \in \W\times\X : e \in x\}.
\end{equation}
Here $\ell_e(m)$ is the \emph{load} on the edge $e$ imposed by the
type-action distribution $m$, which is defined as the fraction of agents using
the element $e$. The cost on a route is additive along  edges, 
and the cost at each edge depends on the corresponding load. 
Notice that the type
does not enter explicitly into $F$, and its only role is to govern the
constraints.

A typical class of examples, representing a traffic network congestion game, originating with the seminal work of Wardrop \cite{wardrop1952road}, is as follows. The set $E$ is the set of edges of some (directed) graph $(V,E)$, so that an action $x \in \X$ is a set of edges. The type space $\W$ is a subset of $V^2$, so that the \emph{type} $w=(i,j)$ of an agent represents the \emph{source} $i$ and  the \emph{destination} $j$ of this agent. The constraint set $\C(w)$ is the set of all (Hamiltonian) paths connecting the source $i$ to the destination $j$, for $w=(i,j)$.

\subsubsection{Existence and uniqueness of equilibrium}
\label{subsub-pot}

Congestion games are well known to belong to the class of
\emph{potential games} \cite{monderer-shapley},  
for which (pure-strategy) Nash equilibria always exist, and for which the uniqueness assumption of Theorem \ref{th:intro-LDP} can be established simply by proving a certain function is strictly convex. Consider the function $U : \P(\W\times\X) \rightarrow \R$ given by
\begin{align}
U(m) = \sum_{e\in E}\int_0^{\ell_e(m)}\!\!\!c_e(s)\,ds. \label{def:U}
\end{align}
Because $c_e \ge 0$ is increasing, the function $t \mapsto
\int_0^tc_e(s)\,ds$ is convex, and thus $U$ is itself convex. Moreover,
recalling the definition of $\A(\lambda)$ from \eqref{def:A_lambda},
it can be shown that for each $\lambda \in \P(\W)$ the set of
minimizers of $U$ on the set $\A(\lambda)$ is precisely $\M(\lambda)$,
the set of Cournot-Nash equilibria with type
distribution $\lambda$.  Hence, when $U$ is
strictly convex, the set $\M(\lambda)$ is a singleton for every
$\lambda$. The following two propositions justify and elaborate on
these claims.  At least the first of the two is well known, but 
we provide the short proofs in Appendix \ref{ap:congestiongames} to
keep the paper self-contained. In the following, $|E|$ denotes the
cardinality of a set $E$ and for a statement $H$, $1_H$ is $1$ if
  the statement $H$ holds and is zero otherwise.

\begin{proposition} \label{pr:congestion-potential}
Fix $\lambda \in \P(\W)$. Then $m$ minimizes $U(\cdot)$ on $\A(\lambda)$ if and only if $m \in \M(\lambda)$.
\end{proposition}

The final proposition, regarding uniqueness of the Cournot-Nash
equilibrium, is likely suboptimal but is merely meant to illustrate
that uniqueness is not an unreasonable request of a congestion game: 

\begin{proposition} \label{pr:congestion-unique}
Enumerate $\W = \{w_1,\ldots,w_{|\W|}\}$ and $\X = \{x_1,\ldots,x_{|\X|}\}$. 
Let $\mathbb{T}$ denote the space of $|\W| \times |\X|$ stochastic matrices, i.e., matrices with nonnegative entries whose columns sum to one.
Assume $c_e$ is differentiable with a strictly positive derivative. Suppose $\lambda \in \P(\W)$ is such that the span of $\{(\lambda\{w_i\}1_{\{e\in x_j\}})_{i,j} : e \in E\}$ contains $\mathbb{T}$. Then $U$ has a unique minimizer on $\A(\lambda)$.
\end{proposition}

\subsubsection{Price of anarchy}
There is a rich literature on \emph{worst-case} bounds, which are
typically valid for a large class of cost functions and 
model specifications. For instance, for the class of linear cost
functions, the seminal paper of 
Roughgarden and Tardos \cite[Theorem 4.5]{roughgarden2002bad} provides a worst-case bound of $4/3$ for the PoA in nonatomic games. More precisely, if $c_e$ is linear for each $e$, then $\mathrm{PoA}(\lambda) \le 4/3$ for all $\lambda \in \P(\W)$.
On the other hand, for  finite games with linear cost functions, Christodoulou and Koutsoupias showed in \cite[Theorem 1]{christodoulou2005price} that the worst-case bound on the PoA is $5/2$. That is, if $c_e$ is linear for each $e$, then $\mathrm{PoA}_n(\vec{w}) \le 5/2$ for all $n$ and all $\vec{w} \in \W^n$.
These PoA bounds are sharp in the sense that there exist
linear cost functions and type distributions for which the bound holds
with equality. Nonetheless,  the following result asserts that for a
\emph{fixed} choice of  linear cost functions $\{c_e\}_{e \in E}$, the 
probability of the PoA in the $n$-player exceeding $4/3$ decays
super-exponentially in $n$.

\begin{corollary} \label{co:PoA-congestion}
In the congestion game model described above, let $\const =
\sup_{\lambda \ll \lambda_0}\mathrm{PoA}(\lambda)$, and assume that
for each $w \in \W$ and every $x \in \C(w)$ there exists $e \in x$
such that $c_e(t) > 0$ for all $t > 0$.  Suppose $\{W_i\}$ is an  i.i.d.\ sequence of  types with distribution
$\lambda_0$. Then, for every $\epsilon > 0$ and $c > 0$, there exists $N$ such that, for all $n \ge N$,
\[
\PP\left(\mathrm{PoA}_n(W_1,\ldots,W_n) \ge \const + \epsilon\right) \le e^{-cn}.
\]
\end{corollary}
\begin{remark}
The assumption in Corollary \ref{co:PoA-congestion} is not very
restrictive; it means that if an admissible route for a given agent
has a nonzero load on every edge, then the route has nonzero travel
time.  This holds, for instance, if $c_e(t) > 0$ for all $t > 0$ and
for all $e \in E$. 
\end{remark}
\begin{proof}[Proof of Corollary \ref{co:PoA-congestion}]
In this model, since $\W$ and $\X$ are finite, using \eqref{def:V} and
\eqref{def-costF}, 
 for any $m \in \P(\W\times\X)$ we can write 
\begin{align*}
V(m) &= \sum_{w \in \W}\sum_{x \in \C(w)}m\{(w,x)\}\sum_{e \in x}c_e(\ell_e(m)).
\end{align*}
Choose $w \in \W$ and $x \in \C(w)$ such that $m\{(w,x)\} > 0$. By assumption, we may find $e \in x$ such that $c_e(t) > 0$ for all $t > 0$. Then
\begin{align*}
\ell_e(m) = \sum_{w' \in \W}\sum_{x' \in \C(w')}1_{e \in x'}m\{(w',x')\} \ge 1_{e \in x}m\{(w,x)\} > 0,
\end{align*}
which implies
\begin{align*}
V(m) &\ge m\{(w,x)\}c_e(\ell_e(m)) > 0.
\end{align*}
We are now in a position to apply Proposition \ref{pr:PoA}.  Because
$\inf\emptyset = \infty$ by convention, 
\begin{align*}
\limsup_{n\rightarrow\infty}\frac{1}{n}\log\,&\PP\left(\mathrm{PoA}_n(W_1,\ldots,W_n) \ge \const + \epsilon\right) \\
	&\le -\inf\left\{H(\lambda | \lambda_0) : \lambda \ll \lambda_0, \ \mathrm{PoA}(\lambda) \ge \const + \epsilon\right\} \\
	&= -\infty.
\end{align*}
\end{proof}

As discussed above, Roughgarden and Tardos showed that the constant
$\const$ of Corollary \ref{co:PoA-congestion} is at most $4/3$ when
$c_e$ is linear for each $e$. Even though the finite $n$-player game
worst-case $\mathrm{PoA}_n$ bound of $5/2$ is optimal \emph{among the class of
  linear cost functions}, our results show that for large $n$, it is
highly unlikely for any \emph{fixed collection of linear cost
  functions $\{c_e\}_{e \in E}$} to produce a
PoA over $4/3$ when sampling i.i.d.\ random types. More generally,
Corollary \ref{co:PoA-congestion}  produces a high-probability PoA bound
for a large but finite population game from 
 a PoA bound for the corresponding class of nonatomic
congestion games.

\section{Extensions and proofs of main results} \label{se:proofs}

We begin our analysis in Section \ref{subs-coupling} by embedding the $n$-player games and the
associated nonatomic  game on a common space, inspired by a
construction of Housman \cite{housman1988infinite}. Then, in Sections 
\ref{subs-pf-lln1} and \ref{subs-pf-ldp} we prove, respectively, the law
of large numbers and large deviation results. Finally, we prove the conditional
limit theorem in Section \ref{se:conditional-limit} and our results on the
price of anarchy in Section \ref{se:PoA-proofs}.

\subsection{A common embedding of $n$-player and nonatomic games}
\label{subs-coupling}

Let $\mathrm{Gr}(\C)$ denote the graph of the constraint set-valued
map $\C$:
\[
\mathrm{Gr}(\C) = \{(w,x) \in \W\times\X : x \in \C(w)\}.
\]
We wish to define an \emph{equilibrium map} $\NN =
\NN(\lambda,\epsilon,u)$, which maps certain 
  elements of $\P(\W) \times [0,\infty)
\times [0,1]$ to subsets of $\P(\W\times\X)$. 
The first input parameter, $\lambda \in \P(\W)$, denotes the distribution of types,
 while the parameter $\epsilon \in [0,\infty)$ 
signifies that we are interested in $\epsilon$-Nash
equilibria (defined precisely in Remark \ref{rem-map}(4) below). 
 Finally, the parameter $u \in [0,1]$ is interpreted as the \emph{size} (or degree of influence) of
an agent. We are only interested in sizes belonging to
$\overline{\N}^{-1} := \{1/n : n =1,2,\ldots\} \cup \{0\}$. When the
size is $1/n$ we are only interested in discrete probability
  distributions of the form $m =
\frac{1}{n}\sum_{i=1}^n\delta_{(w_i,x_i)}$, where $(w_i,x_i) \in
\mathrm{Gr}(\C)$, $i=1, \ldots, n$.  Thus, the domain of the map $\NN$ is 
a certain subset of $\P(\W) \times [0,\infty)
\times [0,1]$, whose definition requires the following notation. 
For any set $\S$, and positive integer $n$, let $\mathcal{E}_{1/n}(\S)
:= \{\frac{1}{n}\sum_{i=1}^n\delta_{e_i} : e_i \in \S\}$ denote the set of
empirical distributions of $n$ points in $\S$. When $\S$ is a
metric space, the convention
$\mathcal{E}_0(\S) := \P(\S)$ will be useful as well, where  as
usual, $\P(\S)$ is the set of Borel probability measures on $\S$. 
Define $\D(\NN)$  to be the set of $(\lambda,\epsilon,u) \in
\P(\W)\times [0,\infty) \times [0,1]$ such that $u \in
\overline{\mathbb{N}}^{-1}$ and $\lambda \in \mathcal{E}_u(\W)$. That
is, 
\begin{align}
\D(\NN) &:= \bigcup_{u \in \overline{\N}^{-1}}\left(\mathcal{E}_{u}(\W) \times [0,\infty) \times \{u\}\right) \label{def-dnn} 
\end{align}
Next, define 
a real-valued function $G$ by 
\begin{align}
G(m,u,w,x) := F(m,w,x)  - \inf_{y \in \C(w)}F\left(m +
  u(\delta_{(w,y)}-\delta_{(w,x)}),w,y\right), \label{def:Gfunction}
\end{align}
for $((m,u),w,x)$ in $\DG \times\W\times\X$, where 
\begin{align}
\DG := \bigcup_{u \in \overline{\N}^{-1}}\left(\mathcal{E}_u(\W\times\X) \times
  \{u\}\right). 
\label{def-dg}
\end{align}
Finally, define the equilibrium map $\NN$ on $\D(\NN)$ by 
\begin{align}
\NN(\lambda,\epsilon,u) &= \left\{m \in \mathcal{E}_u(\W\times\X) :
  m(\text{Gr}(\C)) = 1, \ m^w = \lambda, \  G(m,u,w,x) \le \epsilon
  \text{ for } m\text{-a.e.\ } (w,x)\right\}. \label{def:N_n}
\end{align}

\begin{remark}
\label{rem-map} 
Several comments are in order here. 
\begin{enumerate}
\item 
$\NN(\lambda,0,0)$ is precisely the set $\M(\lambda)$ of Cournot-Nash
equilibria; here the ``error'' parameter $\epsilon$ and the
``size'' parameter $u$ are both  zero, which means that
$\mathcal{E}_0(\W\times\X)=\P(\W\times\X)$ contains all probability
distributions on $\W \times \X$. 
\item  When $u =1/n > 0$ for some positive integer $n$,  there are $n$
  agents, each of ``size'' $1/n$, and $\NN(\lambda,\epsilon,u)$ is a subset of
$\mathcal{E}_u(\W\times\X)$, the empirical distributions of $n$ points in
$\W \times \X$.  
\item 
The term $u(\delta_{(w,y)} - \delta_{(w,x)})$ appearing in $F$ in the
definition \eqref{def:Gfunction} of $G$ accounts for the effect on the
distribution $m$ of agents when  an agent of size $u$ changes its
strategy. 
\item 
 If $(w_1,\ldots,w_n) \in \W^n$ is a type vector for the $n$-player
 game, it is straightforward to see that
 $\NN\left(\frac{1}{n}\sum_{i=1}^n\delta_{w_i},\epsilon,1/n\right)$ is
 precisely the set of empirical distributions
 $\frac{1}{n}\sum_{i=1}^n\delta_{(w_i,x_i)}$, where $(x_1,\ldots,x_n)$
 is an $\epsilon$-Nash equilibrium  with type vector
 $(w_1,\ldots,w_n)$, in the sense that $G(\frac{1}{n}\sum_{i=1}^n\delta_{(w_i,x_i)},1/n,w_i,x_i) \leq \epsilon$ for every $i$.
Most importantly, recalling the definition of $\widehat{\NN}_n(w_1,\ldots,w_n)$ from
\eqref{def:intro-Nn}, we have
\begin{align*}
\NN\left(\frac{1}{n}\sum_{i=1}^n\delta_{w_i},0,\frac{1}{n}\right) = \widehat{\NN}_n(w_1,\ldots,w_n).
\end{align*}
\end{enumerate}
\end{remark}

The key result of this section, inspired by
\cite{housman1988infinite}, 
is that the map $\NN$ is \emph{upper hemicontinuous}, a crucial property that is used in
the proofs of most of the main results.  Let us first
recall some basic definitions regarding set-valued functions.  Let $X$
and $Y$ be topological spaces, and let $\Gamma : X \rightarrow 2^Y$
map points in  $X$ to subsets of $Y$. We say that the set-valued map
$\Gamma$ is \emph{upper hemicontinuous} if $\{x \in X : \Gamma(x) \subset A\}$ is open in $X$ for every open set $A \subset Y$, and we say that $\Gamma$ is \emph{lower hemicontinuous} if $\{x \in X : \Gamma(x) \cap A \neq \emptyset\}$ is open in $X$ for every open set $A \subset Y$. Say that $\Gamma$ is \emph{continuous} if it is both upper and lower hemicontinuous.
 If $Y$ is compact Hausdorff, and if $\Gamma(x)$ is closed for each
 $x$, then $\Gamma$ is upper hemicontinuous if and only if its graph
 $\mathrm{Gr}(\Gamma) = \{(x,y) \in X \times Y : y \in \Gamma(x)\}$ is
 closed \cite[Theorem 17.11]{aliprantisborder}. On the other hand, if
 $X$ and $Y$ are metric spaces, there is a useful sequential
 characterization (c.f.\ Theorems 17.16 and 17.19 of
 \cite{aliprantisborder}): first, $\Gamma$ is lower hemicontinuous if
 and only if, whenever $x_n \rightarrow x$ in $X$ and $y \in
 \Gamma(x)$, there exist integers $1 \le n_1 < n_2 < \ldots$ and
 $y_{n_k} \in \Gamma(x_{n_k})$ such that $y_{n_k} \rightarrow
 y$. Second, a map $\Gamma$ with compact values is upper hemicontinuous if and only if, whenever $x_n \rightarrow x$ in $X$ and $y_n \in \Gamma(x_n)$ for all $n$, the sequence $\{y_n\}$ is precompact, and every limit point belongs to $\Gamma(x)$.

\begin{proposition} \label{pr:Nash-UHC}
The sets $\D(\NN)$ and $\DG$ in \eqref{def-dnn} and \eqref{def-dg} are
closed. The set-valued map $\NN$ in \eqref{def:N_n} is upper
hemicontinuous on $\D(\NN)$ with compact values, and the function $G$
in \eqref{def:Gfunction} 
is continuous.  In particular, $\NN(\lambda,0,0) = \M(\lambda)$ is
closed for all $\lambda \in \P(\W)$. 
\end{proposition}
\begin{proof}
Let $(\lambda_n,\epsilon_n,u_n) \in \D(\NN)$ and
$(\lambda_\infty,\epsilon_\infty,u_\infty) \in \P(\W) \times
[0,\infty) \times [0,1]$ with $(\lambda_n,\epsilon_n,u_n) \rightarrow
(\lambda_\infty,\epsilon_\infty,u_\infty)$.   If $u_\infty=0$, then
trivially $\lambda_\infty \in \mathcal{E}_0(\W) = \P(\W)$, and so
$(\lambda_\infty,\epsilon_\infty,u_\infty) \in \D(\NN)$. If $u_\infty
\neq 0$, then there exists $N$ such that $u_\infty = u_n = u$ for all
$n \ge N$. But then $\lambda_n$ belongs to the closed set
$\mathcal{E}_{u_\infty}(\W)$ for all $n \ge N$, and thus so does
$\lambda_\infty$.  Moreover, by definition $\epsilon_\infty \in
[0,\infty)$.  This shows that $\D(\NN)$ is closed, and the same
argument shows that $\DG$ is closed. 

To show that $\NN$ is upper hemicontinuous,  we use the sequential
characterization described above.  Let $(\lambda_n,\epsilon_n,u_n) \in
\D(\NN)$ with $(\lambda_n,\epsilon_n,u_n) \rightarrow
(\lambda_\infty,\epsilon_\infty,u_\infty)$, and let $m_n \in
\NN(\lambda_n,\epsilon_n,u_n)$ for every $n$. 
First, note that $m_n^w = \lambda_n$ for each $n$, which implies that
$\{m_n^w\} \subset \P(\W)$ is tight by Prokhorov's theorem and our
standing assumption (2)  that $\W$ is a complete separable metric
space.  Because $\X$ is compact,
$\{m_n^x\} \subset \P(\X)$ is also tight. Thus $\{m_n\} \subset \P(\W\times\X)$ is tight, and by Prokhorov's theorem it admits a subsequential limit point $m_\infty$. It remains to show that $m_\infty \in \NN(\lambda_\infty,\epsilon_\infty,u_\infty)$.

We will abuse notation somewhat by assuming $m_n \rightarrow m_\infty$.
Because $m_n(\mathrm{Gr}(\C))=1$ for each $n$ and $\mathrm{Gr}(\C)$ is closed, the Portmanteau theorem yields $m_\infty(\mathrm{Gr}(\C))=1$.
It is clear also that
\[
\lambda_\infty = \lim_{n\rightarrow\infty}\lambda_n =
\lim_{n\rightarrow\infty}m^w_n = m^w_\infty, 
\]
where the limits are in the sense of weak convergence. 
The continuity of $F$ and $\C$ of standing assumptions (3-4) 
implies  the continuity  of $G$ by Berge's theorem
\cite[Theorem 17.31]{aliprantisborder}.  Define measures $\eta_n$ on $\DG \times \W\times\X$ by
\begin{align*}
\eta_n(dm,du,dw,dx) = \delta_{(m_n,u_n)}(dm,du)m_n(dw,dx).
\end{align*}
Then $\eta_n \rightarrow \eta_\infty$  because $(m_n,u_n) \rightarrow (m_\infty,u_\infty)$, and it follows from the
Portmanteau theorem that for any $\Delta  > 0$, 
\begin{align*}
m_\infty\left\{(w,x) : G(m_\infty,u_\infty,w,x) \le \epsilon_\infty+
  \Delta\right\} &= \eta_\infty\left\{(m,u,w,x) : G(m,u,w,x) \le
  \epsilon_\infty + \Delta \right\} \\
	&\ge \limsup_{n\rightarrow\infty}\eta_n\left\{(m,u,w,x) :
          G(m,u,w,x) \le \epsilon_\infty + \Delta \right\} \\
	&= \limsup_{n\rightarrow\infty}m_n\left\{(w,x) :
          G(m_n,u_n,w,x) \le \epsilon_\infty + \Delta \right\} \\
&\geq \limsup_{n\rightarrow\infty} m_n\left\{(w,x) :
          G(m_n,u_n,w,x) \le \epsilon_n \right\} \\
	&= 1,
\end{align*}
where the last inequality uses the fact that, since $\epsilon_n
\rightarrow \epsilon_\infty$, $\epsilon_n \leq \epsilon_\infty +
\Delta$ for all sufficiently large $n$, and the last equality  holds
because  $m_n \in \NN(\lambda_n,\epsilon_n,u_n)$.   Since 
$\Delta > 0$ is arbitrary,  it follows that  $G(m_\infty,u_\infty,w,x) \le
\epsilon_\infty$ for $m_\infty$ a.e.\ $(w,x)$. 
It remains to check that $m_\infty$ belongs to
$\mathcal{E}_{u_\infty}(\W\times\X)$. If $u_\infty=0$, there is
nothing to prove because of the convention $\mathcal{E}_0(\W\times\X)
= \P(\W\times\X)$. If $u_\infty > 0$, then there exists $N$ such that
$u_n=u_\infty$ for all $n \ge N$. Then  $m_n$ belongs to the closed
set $\mathcal{E}_{u_\infty}(\W\times\X)$ for all $n \ge N$, and hence,
so does $m_\infty$.
\end{proof}

\subsection{Existence of Cournot-Nash equilibria} \label{se:existence}

Under our standing assumptions, there always exist
Cournot-Nash equilibria for the nonatomic game.  The proof uses a
simple argument due to Mas-Colell \cite{mas1984theorem}, 
appropriately modified  to incorporate the constraint map $\C$.

\begin{proposition} \label{pr:existence}
For each $\lambda \in \P(\W)$, $\M(\lambda) \neq \emptyset$.
\end{proposition}
\begin{proof}
Let $\mathrm{Gr}(\C) = \{(w,x) \in \W\times\X : x \in \C(w)\}$ as before, and define $\A(\lambda)$ as in \eqref{def:A_lambda}.
Note that $\A(\lambda)$ is closed, as $\mathrm{Gr}(\C)$ is closed by
our 
standing assumption (3). 
Because $\X$ is compact and $\W \times \X$ is a complete
  separable metric space by  our standing assumptions (1-2), it is
  straightforward to check that $\A(\lambda)$ is tight and thus
  compact. Consider the map $\Phi$ from $\A(\lambda)$ into subsets of $\A(\lambda)$, given by
\[
\Phi(m) = \left\{\widetilde{m} \in \A(\lambda) :
	\int_{\W \times \X}G(m,0,w,x)\widetilde{m}(dw,dx) \le 0 
	\right\}
\]
for $m \in \A(\lambda)$. 
Note that $m \in \P(\W\times\X)$ is a Cournot-Nash equilibrium with type distribution $\lambda$ if and only if $m \in \Phi(m)$, i.e., $m$ is a fixed point of $\Phi$.
Clearly  $\A(\lambda)$ is convex, and hence $\Phi(m)$ is  convex for each $m$.
The graph $\mathrm{Gr}(\Phi) = \{(m,\widetilde{m}) \in
\A(\lambda)\times\A(\lambda) : \widetilde{m} \in \Phi(m)\}$ is easily
seen to be closed, using the fact that $G$ is continuous (due to
Proposition \ref{pr:Nash-UHC}) and bounded (by standing assumption
(4)). To check that $\Phi(m)$ is nonempty for each $m$, note that there exists (e.g., by \cite[Theorem 18.19]{aliprantisborder}) a measurable function $\hat{x} : \A(\lambda) \times \W \rightarrow \X$ such that $\hat{x}(m,w) \in \C(w)$ and $F(m,w,\hat{x}(m,w))=\inf_{y \in \C(w)}F(m,w,y)$ for each $(m,w) \in \A(\lambda) \times \W$. Then $\hat{m}(dw,dx) = \lambda(dw)\delta_{\hat{x}(m,w)}(dx)$ always belongs to $\Phi(m)$. Because $\Phi$ has a closed graph and nonempty convex values, it admits a fixed point by Kakutani's theorem \cite[Corollary 17.55]{aliprantisborder}.
\end{proof}

\subsection{Proof of laws of large numbers}
\label{subs-pf-lln1}

Using Proposition \ref{pr:Nash-UHC}, we give streamlined proofs of
Theorems \ref{th:intro-limit-setvalued} and \ref{th:intro-limit}, and
even an extension of the latter. 

\begin{proof}[Proof of Theorem \ref{th:intro-limit-setvalued}]
Proposition \ref{pr:Nash-UHC} shows that $\NN$ is upper hemicontinuous. According to Lemma \ref{le:uhcVietoris}, this implies that $\NN$ is continuous as a map from $\D(\NN)$ to the space $\mathfrak{C}$ of closed subsets of $\P(\W\times\X)$ endowed with the upper Vietoris topology. Because $(\frac{1}{n}\delta_{W_i},0,\frac{1}{n})$ converges almost surely to $(\lambda_0,0,0)$, it follows (see Remark \ref{rem-map}(4) for the first equality) that, almost surely,
\begin{align*}
\widehat{\NN}_n(W_1,\dots,W_n) &= \NN\left(\frac{1}{n}\delta_{W_i},0,\frac{1}{n}\right) \rightarrow \NN(\lambda_0,0,0) = \M(\lambda_0).
\end{align*}
\end{proof}

We next turn to the proof of Theorem \ref{th:intro-limit}, which
we  precede with a reassuring technical lemma.
First, for $\epsilon \ge 0$ and $n \in
\mathbb{N}$, let  $N^\epsilon_n (\vec{w})$ denote the set of $\epsilon$-Nash
equilibria with type vector $\vec{w} \in \W^n$.  By Remark
\ref{rem-map}(4), this can be expressed in  terms of the
function $G$ of \eqref{def:Gfunction} as  
\[  N^\epsilon_n (w_1, \ldots, w_n) = \left\{ \vec{x} \in \X^n: x_i \in
\C(w_i) \mbox{ and }  G\left(\frac{1}{n} \sum_{i=1}^n
\delta_{(w_i, x_i)}, \frac{1}{n}, w_i, x_i\right)\leq \epsilon, \, \forall
i= 1, \ldots, n\right\}. 
\] 
Also, let $D^\epsilon_n$ be the set of $\vec{w} \in \W^n$ for which there exists an
associated $\epsilon$-Nash equilibrium: 
\[ D^\epsilon_n = \left\{ \vec{w} \in \W^n: N^\epsilon_n (\vec{w}) 
  \neq \emptyset\right\}. 
\]

\begin{lemma} \label{le:measurable-selection}
For each $n$ and $\epsilon \ge 0$, the set  $D^\epsilon_n$ is closed.
Moreover, there exists a universally measurable map $\hat{x} :
D^\epsilon_n \rightarrow \X^n$ such that $\hat{x}(\vec{w}) \in
N^\epsilon_n(\vec{w})$ for each $\vec{w} \in D^\epsilon_n$. 
\end{lemma}
\begin{proof}
Continuity of $G$ (proven in Proposition \ref{pr:Nash-UHC}) and closedness of the graph of $\C$ (one of our standing assumptions) together imply that the graph
\[
\mathrm{Gr}(N^\epsilon_n) = \left\{(\vec{w},\vec{x}) \in \W^n \times \X^n : \vec{x} \in N^\epsilon_n(\vec{w})\right\}
\]
 is closed. The projection from $\W \times \X$ to $\W$ is a closed map, since $\X$ is compact, which shows that $D^\epsilon_n$ is closed. The existence of $\hat{x}$ follows from the Jankov-von Neumann theorem \cite[Proposition 7.49]{bertsekasshreve}.
\end{proof}

\begin{theorem} \label{th:limit}
Let $\epsilon_n \ge 0$ be such that  $\epsilon_n \rightarrow 0$.
Suppose, for each $n$ and each $\vec{w} \in \W^n$, we are given an $\epsilon_n$-Nash equilibrium $\hat{x}^n(\vec{w})=(\hat{x}^n_1(\vec{w}),\ldots,\hat{x}^n_n(\vec{w}))$ with type vector $\vec{w}$. By Lemma \ref{le:measurable-selection} we may assume each $\hat{x}^n_i$ is universally measurable. Finally, suppose $\vec{W}^n=(W^n_1,\ldots,W^n_n)$ is a $\W^n$-valued random vector, and define the random empirical distributions
\[
\mu_n = \frac{1}{n}\sum_{i=1}^n\delta_{(W^n_i,\hat{x}^n_i(\vec{W}^n))}, \quad\quad\quad \mu^w_n := \frac{1}{n}\sum_{i=1}^n\delta_{W^n_i} .
\]
Then the  following hold:
\begin{enumerate}[(i)]
\item If the sequence $\{\mu_n^w\}$ is tight, then so is the sequence
  $\{\mu_n\}$, and every subsequential limit in distribution $\mu$ of $\{\mu_n\}$ satisfies $\mu \in \M(\mu^w)$, almost surely.
\item If $\mu^w_n \rightarrow \lambda_0$ in distribution, where
  $\lambda_0 \in \P(\W)$ is deterministic, then every subsequential
 limit in  distribution of $\{\mu_n\}$ is supported on
  $\M({\lambda_0})$. 
In particular, 
\begin{align}
\lim_{n\rightarrow\infty}\PP(d(\mu_n,\M ({\lambda_0})) \ge \epsilon) = 0, \label{def:convergence-in-prob}
\end{align}
for every $\epsilon > 0$, and any metric $d$ on $\P(\W\times\X)$
compatible with weak convergence.
\item  
  If $W^n_1,\ldots,W^n_n$ are i.i.d.\ with distribution $\lambda_0 \in
  \P(\W)$, for each $n$, then $d(\mu_n,\M({\lambda_0})) \rightarrow 0$
  almost surely, for $d$ as in (ii). 
\end{enumerate}
\end{theorem}
\begin{proof} {\ } \\
\begin{enumerate}[(i)]
\item By \cite[Proposition 2.2(ii)]{sznitman}, tightness of $\{\PP
  \circ (\mu_n^w)^{-1}\} \subset \P(\P(\W))$ is equivalent to
  tightness of the sequence of mean measures $\{\E[\mu_n^w]\} \subset
  \P(\W)$, where $\E[\mu_n^w](\cdot) := \E[\mu_n^w(\cdot)]$. The
  mean measure $\E[\mu_n]$ has first marginal $\E[\mu_n^w]$, and
  because $\X$ is compact we conclude that $\{\E[\mu_n]\} \subset \P(\W\times\X)$ is tight. Again using \cite[Proposition 2.2(ii)]{sznitman}, we conclude that $\{\PP \circ \mu_n^{-1}\} \subset \P(\P(\W\times\X))$ is tight.
Now, by Skorohod's representation theorem, we may assume that (along a
subsequence) $\mu_n$ converges almost surely to a random element $\mu$
of $\P(\W\times\X)$. This implies $\mu^w_n \rightarrow \mu^w$
a.s. Since $\mu_n \in \NN(\mu^w_n,\epsilon_n,1/n)$ by assumption,
the upper hemicontinuity of $\NN$ implies that a.s., $\mu$ must belong to
$\NN(\mu^w,0,0) = \M(\mu^w)$, where the last equality holds by
Remark \ref{rem-map}(1). 
\item Suppose the random measure $\mu$ is a subsequential limit in distribution of $\{\mu_n\}$. Because
  $\mu^w_n \rightarrow \lambda_0$, we must have $\mu^w=\lambda_0$
  a.s.   We conclude from (i) that $\PP(\mu
  \in \M(\lambda_0))=1$.  Thus, for any subsequential limit $\mu$ of
  $\{\mu_n\}$  and    $\epsilon > 0$, we have 
  $\PP(d(\mu,\M(\lambda_0)) \ge \epsilon) = 0$. 
When combined with the Portmanteau theorem, the closedness of the set
  $\M(\lambda_0)$ established in Proposition \ref{pr:Nash-UHC}  and the
  consequent   closedness of $\{m \in \P(\W\times\X) :
  d(m,\M(\lambda_0)) \ge \epsilon\}$ imply the claim
  \eqref{def:convergence-in-prob}. 
\item  Almost surely, the following holds: Because  $\mu^w_n \rightarrow
  \lambda_0$ due to $\{W_i^n\}$ being i.i.d.,  and $\mu_n \in \NN(\mu^w_n,\epsilon_n,1/n)$, upper
  hemicontinuity of $\NN$ implies that that the limit
  $\lim_{k\rightarrow\infty}\mu_{n_k}$ exists along some subsequence,
  and every such limit belongs to $\NN(\lambda_0,0,0)$.  By Remark
  \ref{rem-map}(1), this is enough to show $d(\mu_n,\M(\lambda_0)) \rightarrow 0$.
\end{enumerate}
\end{proof}

\subsection{Proofs of large deviation results}
\label{subs-pf-ldp}

We are now prepared to state and prove an extension of our main
result (Theorem \ref{th:intro-LDP-setvalued}) that allows for
approximate equilibria.   
Recall from Section \ref{se:intro:LDP-set} the definition of the space $\mathfrak{C}$, equipped with the
upper Vietoris topology.  Having identified the suitable space,
topology and mappings,   the proof of this extension  follows from  a simple application of the contraction
principle from large deviations theory. 
As we will use it on several occasions, it is worth recalling here the
general definition of an LDP.  
We say that a
sequence of Borel probability measures $\{\nu_n\}$ on a topological space
$S$ satisfies an LDP  with good rate function $I :
S \rightarrow [0,\infty]$ if the level set $\{s \in S : I(s) \le c\}$
is compact for each $c \ge 0$ and if the following holds for every
Borel set $A \subset S$: 
\begin{align*}
\limsup_{n\rightarrow\infty}\frac{1}{n}\log \nu_n(A) &\le -\inf_{s \in \overline{A}}I(s), \\
\liminf_{n\rightarrow\infty}\frac{1}{n}\log \nu_n(A) &\ge -\inf_{s \in A^\circ}I(s),
\end{align*}
where $\overline{A}$ and $A^\circ$ denote the closure and interior. We
say a sequence of $S$-valued random variables satisfies an LDP  
 if the corresponding sequence of probability measures does. In the following, recall the definition of the relative entropy $H$ from \eqref{def:entropy}.

\begin{theorem} \label{th:LDP-setvalued}
Suppose $\epsilon_n \rightarrow 0$ and $\{W_i\}$ is an
i.i.d.\ sequence of $\W$-valued random variables with common type
distribution $\lambda_0$.  Then the sequence of sets of $\epsilon_n$-Nash equilibria
\[
\NN\left(\frac{1}{n}\sum_{i=1}^n\delta_{W_i},\epsilon_n,\frac{1}{n}\right),
\quad n \in \mathbb{N}, 
\]
satisfies an LDP 
 on $\mathfrak{C}$ with good rate function 
\begin{align}
J(A) = \inf\left\{H(\lambda | \lambda_0) : \lambda \in \P(\W), \
  \M(\lambda) = A\right\}. \label{def:J}
\end{align}
\end{theorem}
\begin{proof}
First recall from Proposition \ref{pr:Nash-UHC} that
$\NN(\lambda,\epsilon,u)$ is closed and thus belongs to
$\mathfrak{C}$, for every $(\lambda,\epsilon,u) \in \D(\NN)$. 
Define two $\D(\NN)$-valued random variables
\[
M_n = \left(\frac{1}{n}\sum_{i=1}^n\delta_{W_i},\epsilon_n,\frac{1}{n}\right), \quad\quad\quad M_n^0 = \left(\frac{1}{n}\sum_{i=1}^n\delta_{W_i},0,0\right).
\] 
 By Sanov's theorem and the contraction principle \cite[Theorem 4.2.1]{dembozeitouni},
applied to the continuous map $\P(\W) \ni \lambda \mapsto
(\lambda,0,0) \in \P(\W)\times [0,\infty)\times [0,1]$, 
$\{M^0_n\}$ satisfies an LDP 
on $\P(\W)\times [0,\infty)\times [0,1]$ with good rate function
\begin{align*}
(\lambda,\epsilon,u) &\mapsto \begin{cases}
H(\lambda | \lambda_0) &\text{if } \epsilon=u=0, \\
\infty &\text{otherwise}.
\end{cases}
\end{align*}
The sequences $\{M_n\}$ and $\{M^0_n\}$ are exponentially
equivalent in the sense that (c.f.\ \cite[4.2.10]{dembozeitouni})
\begin{align}
\limsup_{n\rightarrow\infty}\frac{1}{n}\log\PP(\bar{d}(M_n, M_n^0) \ge a) = -\infty, \text{ for all } a > 0, \label{pf:exp-equiv}
\end{align}
where we define the metric $\bar{d}$ on $\P(\W)\times [0,\infty)\times [0,1]$ by 
\[
\bar{d}((\lambda',\epsilon',u'),(\lambda,\epsilon,u)) = d(\lambda,\lambda') + |\epsilon-\epsilon'| + |u-u'|,
\]
where $d$ is any metric on $\P(\W)$ compatible with weak
convergence. In fact, the probability in \eqref{pf:exp-equiv} 
is zero for sufficiently large $n$. 
Thus, $\{M_n\}$ satisfies an LDP 
with the same rate function \cite[Theorem 4.2.13]{dembozeitouni}.  Because $\NN$ is
upper hemicontinuous as a set-valued map (by Proposition \ref{pr:Nash-UHC}), it is continuous as a map
from $\D(\NN)$ to $\mathfrak{C}$,
equipped with the upper Vietoris topology (see Lemma
\ref{le:uhcVietoris}). Thus, the contraction principle (see
\cite[Theorem 4.2.1]{dembozeitouni}, which does not actually need the
spaces to be Hausdorff) implies that
$\{\NN(M_n)=\NN\left(\frac{1}{n}\sum_{i=1}^n\delta_{W_i},\epsilon_n,\frac{1}{n}\right)\}$ 
satisfies  an LDP 
on $\mathfrak{C}$ with good rate
function
\begin{align*} 
A &\mapsto \inf\left\{H(\lambda | \lambda_0) : \lambda \in \P(\W), \ \NN(\lambda,0,0) = A\right\}  \\
	&= \inf\left\{H(\lambda | \lambda_0) : \lambda \in \P(\W), \ \M(\lambda) = A\right\} \\
	&= J(A),
\end{align*}
where the first equality uses the fact that $\NN (\lambda, 0, 0) = \M
(\lambda)$ from Remark \ref{rem-map}(4). 
\end{proof}

\begin{remark}
\label{rem-LDP-setvalued}
Recalling from Remark \ref{rem-map}(4) that 
$\widehat{\NN}_n(w_1,\ldots,w_n) =
\NN\left(\frac{1}{n}\sum_{i=1}^n\delta_{w_i},0,1/n\right)$, Theorem
\ref{th:intro-LDP-setvalued} is an immediate corollary of Theorem
\ref{th:LDP-setvalued}. 
\end{remark}

\begin{remark}
Because the proof of Theorem \ref{th:LDP-setvalued} relies on the
contraction principle, a similar result holds if we weaken the
assumptions on the type sequence $\{W_i\}$. They need not be i.i.d.,
as long as the sequence of empirical distributions
$\frac{1}{n}\sum_{i=1}^n\delta_{W_i}$ satisfies some  LDP. 
\end{remark}

We next state an extension of Theorem \ref{th:intro-LDP} and prove it
using Theorem \ref{th:LDP-setvalued} and some elementary properties of
the upper Vietoris topology. Interestingly, even without uniqueness we
find upper and lower bounds, although they do not match in general.

\begin{theorem} \label{th:LDP}
Use the notation and assumptions of Theorem \ref{th:limit}, and assume
also that 
$W^n_i=W_i$ for  $1 \le i \le n$, where $\{W_i\}$ is an
  i.i.d.\ sequence with distribution $\lambda_0 \in \P(\W)$. Then we have the following bounds, valid for every measurable set $A \subset \P(\W\times\X)$:
\begin{align}
\limsup_{n\rightarrow\infty}\frac{1}{n}\log\PP(\mu_n \in A) &\le -\inf\left\{H(m^w|\lambda_0) : m \in \overline{A} \cap \M\right\}, \label{th:def:LDP-upperbound} \\
\liminf_{n\rightarrow\infty}\frac{1}{n}\log\PP(\mu_n \in A) &\ge -\inf\{H(\lambda|\lambda_0) : \lambda \in \P(\W), \ \M(\lambda) \subset A^\circ\}. \label{th:def:LDP-lowerbound} 
\end{align}
Moreover, if $\M(\lambda)$ is a singleton for every $\lambda \ll \lambda_0$, then
\begin{align}
\liminf_{n\rightarrow\infty}\frac{1}{n}\log\PP(\mu_n \in A) &\ge -\inf\left\{H(m^w|\lambda_0) : m \in A^\circ \cap \M\right\}. \label{th:def:LDP-lowerbound2} 
\end{align}
\end{theorem}
\begin{proof}
Suppose $A$ is closed. Then $\mathfrak{U} := \{E \in \mathfrak{C} : E \subset A^c\}  = \{E \in \mathfrak{C} : E \cap A = \emptyset\}$ is open in the upper Vietoris topology, so its complement is closed. Thus, using the upper bound of Theorem \ref{th:LDP-setvalued},
\begin{align*}
\limsup_{n\rightarrow\infty}&\frac{1}{n}\log\PP\left(\NN\left(\frac{1}{n}\sum_{i=1}^n\delta_{W_i},\epsilon_n,\frac{1}{n}\right) \cap A \neq \emptyset\right) \\
	&\le -\inf_{B \in \mathfrak{U}^c}J(B) \\
	&= -\inf_{B \in \mathfrak{U}^c}\inf\left\{H(\lambda | \lambda_0) : \lambda \in \P(\W), \ \M(\lambda) = B\right\} \\
	&= -\inf\left\{H(\lambda | \lambda_0) : \lambda \in \P(\W), \ \M(\lambda) \cap A \neq \emptyset\right\} \\
	&= -\inf\left\{H(m^w | \lambda_0) : m \in \M \cap A\right\}.
\end{align*}
Indeed, this last equality follows from two simple observations: If
$\M(\lambda) \cap A \neq \emptyset$, then there exists $m \in \M
(\lambda) \cap
A \subset \M \cap A$ such that $m^w = \lambda$. On the other hand, if $m \in \M \cap A$, then $m \in \M(m^w)$, so $\M(m^w)\cap A \neq \emptyset$. Finally, the upper bound \eqref{th:def:LDP-upperbound} follows from the inequality
\[
\PP\left(\mu_n \in A\right) \le \PP\left(\NN\left(\frac{1}{n}\sum_{i=1}^n\delta_{W_i},\epsilon_n,\frac{1}{n}\right) \cap A \neq \emptyset\right),
\]
which holds because, by Remark \ref{rem-map}(4), $\mu_n \in \NN\left(\frac{1}{n}\sum_{i=1}^n\delta_{W_i},\epsilon_n,\frac{1}{n}\right)$ a.s.

To prove the lower bound, let $A$ be open, and notice that $\mathfrak{U} = \{E
\in \mathfrak{C} : E \subset A\}$ is open in the upper Vietoris
topology. Theorem \ref{th:LDP-setvalued}  then  implies
\begin{align*}
\liminf_{n\rightarrow\infty}&\frac{1}{n}\log\PP\left(\NN\left(\frac{1}{n}\sum_{i=1}^n\delta_{W_i},\epsilon_n,\frac{1}{n}\right) \subset A\right) \\
	&\ge -\inf_{B \in \mathfrak{U}}J(B) \\
	&= -\inf_{B \in \mathfrak{U}}\inf\left\{H(\lambda | \lambda_0) : \lambda \in \P(\W), \ \M(\lambda) = B\right\} \\
	&= -\inf\left\{H(\lambda | \lambda_0) : \lambda \in \P(\W), \
          \M(\lambda) \subset A\right\}, 
\end{align*} 
where the last equality uses the property that $\M (\lambda)$ is
closed for every $\lambda \in \P(\W)$ (see Proposition
\ref{pr:Nash-UHC}). 
Then the lower bound \eqref{th:def:LDP-lowerbound} follows from the inequality
\[
\PP(\mu_n \in A) \ge \PP\left(\NN\left(\frac{1}{n}\sum_{i=1}^n\delta_{W_i},\epsilon_n,\frac{1}{n}\right) \subset A\right),
\]
which again holds because $\mu_n \in
\NN\left(\frac{1}{n}\sum_{i=1}^n\delta_{W_i},\epsilon_n,\frac{1}{n}\right)$
a.s.; see Remark \ref{rem-map}(4). 

Finally, we deduce \eqref{th:def:LDP-lowerbound2} from \eqref{th:def:LDP-lowerbound}. Again let $A$ be open and note first that $H(\lambda | \lambda_0) < \infty$ only if $\lambda \ll \lambda_0$. Supposing $\M(\lambda) = \{M[\lambda]\}$ is a singleton for all $\lambda \ll \lambda_0$, then trivially $m= M[m^w]$ for all $m \in \M$, and
\begin{align*}
\inf\left\{H(\lambda | \lambda_0) : \lambda \in \P(\W), \ \M(\lambda) \subset A\right\} &= \inf\left\{H(\lambda | \lambda_0) : \lambda \ll \lambda_0, \ M[\lambda] \in A\right\} \\
	&= \inf\left\{H(m^w | \lambda_0) : m \in \M \cap A, \ m^w \ll \lambda_0\right\}  \\
	&=\inf\left\{H(m^w | \lambda_0) : m \in \M \cap A\right\} .
\end{align*}
\end{proof}

In applications, it is important to know if the bounds in the large
deviation principles of Theorems \ref{th:LDP} and
\ref{th:LDP-setvalued} are nonzero. The following straightforward lemmas help to check this.

\begin{lemma} \label{le:inf>0-setvalued}
Let $J$ be as in \eqref{def:J}, and let $\mathfrak{U} \subset \mathfrak{C}$ be a
closed set with $\M(\lambda_0) \notin \mathfrak{U}$. Then $\inf_{A \in \mathfrak{U}}J(A) > 0$.
\end{lemma}
\begin{proof}
Note that
\[
\inf_{A \in \mathfrak{U}}J(A) = \inf\left\{H(\lambda | \lambda_0) : \lambda \in \P(\W), \ \M(\lambda) \in \mathfrak{U}\right\}.
\]
Recall that $\lambda \mapsto \M(\lambda) = \NN(\lambda,0,0) \in
\mathfrak{C}$ is continuous by Proposition \ref{pr:Nash-UHC} and Lemma
\ref{le:uhcVietoris}. Hence, the set $S = \{\lambda \in \P(\W) :
\M(\lambda) \in \mathfrak{U}\}$ is closed because $\mathfrak{U}$ is.  Because $\lambda \mapsto H(\lambda | \lambda_0)$ is lower semicontinuous and has compact sub-level sets, there exists $\lambda^* \in S$ such that $H(\lambda^* | \lambda_0) =  \inf_{A \in \mathfrak{U}}J(A)$. But $\M(\lambda^*) \in \mathfrak{U}$ implies $\lambda^* \neq \lambda_0$, and thus $H(\lambda^* | \lambda_0) > 0$.
\end{proof}

For our second observation, Lemma \ref{le:inf>0} below, we  need the following simple property:

\begin{lemma} \label{le:Mclosed}
The set $\M$ is closed.  Moreover,  
 the sub-level set $\{m \in \P(\W\times\X) : H(m^w | \lambda_0) \le
 c\}$ is  compact for every $c < \infty$. 
\end{lemma}
\begin{proof}
Suppose $m_n \in \M$ and $m \in \P(\W\times\X)$ with $m_n \rightarrow
m$. Then $\lambda_n := m_n^w$ converges to $\lambda := m^w$.  Then
$m_n \in \M(\lambda_n)= \NN(\lambda_n,0,0)$, and the upper
hemicontinuity of $\NN$ (proven in Proposition \ref{pr:Nash-UHC})
implies that the unique limit point $m$ must belong to
$\NN(\lambda,0,0) \subset \M$.  This proves that $\M$ is closed. 
The second statement follows because $\X$ is compact and  the sub-level set $\{\lambda \in \P(\W) : H(\lambda | \lambda_0) \le
c\}$ is compact for each $c < \infty$  \cite[Lemma 1.4.3(c)]{Dupuis-Ellis}. 
\end{proof}

\begin{lemma} \label{le:inf>0}
If $A \subset \P(\W\times\X)$ is closed and $A \cap \M(\lambda_0) =
\emptyset$, then $\inf_{m \in A \cap \M}H(m^w|\lambda_0) > 0$. 
\end{lemma}
\begin{proof}  Since $\M$ is closed and the sublevel sets 
$\{m:  H(m^w|\lambda_0) \le c\}$ are compact 
by Lemma \ref{le:Mclosed}, $A \cap
  \M$ is closed and there exists   $m_* \in A \cap \M$ such that
 $H(m_*^w | \lambda_0) =  \inf_{m \in A \cap \M} H(\lambda|\lambda_0)$.
But $A \cap \M(\lambda_0) = \emptyset$ implies $m_*^w \neq \lambda_0$,
and thus, $H(m_*^w | \lambda_0) > 0$. 
\end{proof}

\subsection{The conditional limit theorem and entry games} \label{se:conditional-limit}

In this section we first prove Theorem 
\ref{th:intro-conditional-limit}. Then, to illustrate the tractability
of the assumptions, we apply the theorem to an example from the class
of entry games discussed in Section \ref{subs-intro-mot}.

\begin{proof}[Proof of Theorem \ref{th:intro-conditional-limit}]
Fix a measurable set $A \subset {\mathcal P}( \W \times \X)$.  
Since $I(A) < \infty$,  the closedness of $\M$ and the compactness of the sub-level sets 
 $\{m \in \P(\W\times\X) : H(m^w | \lambda_0) \le
 c\}$ established in  Lemma \ref{le:Mclosed} imply 
that the set  $S(A)$ of minimizers in \eqref{def:S(A)} is non-empty
and compact. 
Next, use  the lower bound of  Theorem \ref{th:LDP} to get
\begin{align*}
\liminf_{n\rightarrow\infty}\frac{1}{n}\log\PP\left(\mu_n \in A\right) &\ge -\inf\left\{H(\lambda | \lambda_0) : \lambda \in \P(\W), \ \M(\lambda) \subset A^0\right\} \\
	&= -I(A) \\
	&= -\inf\left\{H(\lambda | \lambda_0) : \lambda \in \P(\W), \ \M(\lambda) \cap \overline{A} \neq \emptyset\right\} \\
	&= -\inf\left\{H(m^w | \lambda_0) : m \in \M \cap \overline{A}\right\},
\end{align*}
where we have used the assumption \eqref{def:conditional-assumption}
in the third line. Because the set
\[
A_\epsilon := \{m \in A : d(m,S(A)) \ge \epsilon\}
\]
is closed,  the upper bound of  Theorem \ref{th:LDP} yields
\begin{align*}
\limsup_{n\rightarrow\infty}&\frac{1}{n}\log\PP\left(\left. d(\mu_n,S(A)) \ge \epsilon \right| \mu_n \in A\right) \\
	&= \limsup_{n\rightarrow\infty}\frac{1}{n}\log\PP\left( \mu_n \in A_\epsilon\right) - \liminf_{n\rightarrow\infty}\frac{1}{n}\log\PP\left(\mu_n \in A\right) \\
	&\le \inf_{m \in \overline{A} \cap \M}H(m^w | \lambda_0) - \inf_{m \in A_\epsilon \cap \M}H(m^w|\lambda_0) \\
& =: C, 
\end{align*}
Then clearly $C \le 0$, as $\overline{A} \supset A_\epsilon$. If
$C=0$, then there exists $m \in A_\epsilon \cap \M$ such
that $H(m^w|\lambda_0) = \inf_{m \in \overline{A} \cap \M}H(m^w | \lambda_0)$. But this implies $m \in S(A)$, which contradicts the fact that $S(A)$ and $A_\epsilon$ are
disjoint. Thus $C < 0$, and the proof of \eqref{decay} is complete.

Finally, if $\M(\lambda) = \{M(\lambda)\}$ is a singleton for every $\lambda \ll \lambda_0$, then the identity $(M(\lambda))^w = \lambda$ implies
\begin{align*}
\inf\left\{H(\lambda | \lambda_0) : \lambda \ll \lambda_0, \ \M(\lambda) \subset A^\circ\right\} &= \inf\left\{H(\lambda | \lambda_0) : \lambda \ll \lambda_0, \ M(\lambda) \in A^\circ\right\} \\
	&= \inf\left\{H(m^w | \lambda_0) : m \in A^\circ \cap \M, \ m^w \ll \lambda_0\right\} \\
	&= \inf\left\{H(m^w | \lambda_0) : m \in A^\circ \cap \M \right\}.
\end{align*}
Similarly, we noted above already that 
\[
I(A) = \inf\left\{H(\lambda | \lambda_0) : \lambda \ll \lambda_0, \ \M(\lambda) \cap \overline{A} \neq \emptyset\right\} = \inf\left\{H(m^w | \lambda_0) : m \in \M \cap \overline{A}\right\}.
\]
\end{proof}

Let us now consider a simple entry game, specified as follows. There are two types and two actions, with $\W =
\{1,2\}$ and $\X = \{0,1\}$, there are no constraints in this
  model, so $\C(w)=\X$ for all $w \in \W$, and
 the objective function is given by
\[
F(m,w,x) = -x(-3m^x\{1\} + w), 
\]
for $m \in \P(\W \times \X)$, $w \in \W$, and $x \in \X$.
Think of each agent as facing a fixed payoff $w$ from entering the
market, i.e., choosing $x=1$. This payoff is offset by a loss of
$3 m^x\{1\}$ which increases with the fraction of agents entering the
market. If the net payoff is negative, the agent will choose $x=0$
and not enter the market.

For $q \in [0,1]$, let $\lambda_q = q\delta_{2} + (1-q)\delta_{1}$ denote the type distribution in which the fraction of type-$2$ agents is $q$. Of course, $\{\lambda_q : q \in [0,1]\}$ exhausts all of $\P(\W)$. To apply our conditional limit theorem, we first characterize all possible Cournot-Nash equilibria:

\begin{proposition} \label{pr:entrygame}
For the entry game described above, for each $q \in [0,1]$ the
Cournot-Nash equilibrium is unique. That is, $\M(\lambda_q) =
\{m_q\}$, where $m_q  \in {\mathcal P} (\W \times \X)$ is defined by
\begin{align*}
\left(\begin{matrix}
m_q\{(1,0)\} & m_q\{(1,1)\} \\
m_q\{(2,0)\} & m_q\{(2,1)\}
\end{matrix}\right) = \begin{cases}
\left(\begin{matrix}
2/3 & 1/3-q \\
0 & q
\end{matrix}\right) &\text{if } q \le 1/3,\vspace{0.4em} \\
\left(\begin{matrix}
1-q & 0 \\
0 & q
\end{matrix}\right) &\text{if } 1/3 < q < 2/3,\vspace{0.4em} \\
\left(\begin{matrix}
1-q & 0 \\
q-2/3 & 2/3
\end{matrix}\right) &\text{if } q \ge 2/3.
\end{cases}
\end{align*}
\end{proposition}
\begin{proof}
We know $\M(\lambda_q) \neq \emptyset$ for each $q \in [0,1]$, thanks to 
Proposition \ref{pr:existence}. Fix $q \in [0,1]$ and $m \in \M(\lambda_q)$, and abbreviate $p=m^x\{1\}$. We will show that $m=m_q$. Note first that 
\begin{align*}
\arg\min_{x \in \{0,1\}}F(m,w,x) = \begin{cases}
\{1\} &\text{if } w > 3p, \\
\{0\} &\text{if } w < 3p,  \\
\{0,1\} &\text{if } w=3p.
\end{cases}
\end{align*}
Next, there are three cases to check. First, if $3p \notin \{1,2\}$, then $m\{(w,1)\}
= 1_{\{w > 3p\}}$ for each $w$. Thus, 
\[
p = m^x\{1\} = q1_{\{2 > 3p\}} + (1-q)1_{\{1 > 3p\}} = \begin{cases}
0 &\text{if } p > 2/3 \\
q &\text{if } 2/3 > p > 1/3 \\
1 &\text{if } p < 1/3.
\end{cases}
\]
This can only hold if $p=q$ and $1/3 < p < 2/3$. For the second case,
suppose $p=1/3$. Then all type-$2$ agents enter since $2 > 3p$; that is,
$m\{(2,1)\} = q$ and  $m\{(2,0)\}=0$. Therefore, we have
\[
1/3 = p = m\{(1,1)\} + m\{(2,1)\} = m\{(1,1)\} + q, 
\] 
which implies $m\{(1,1)\} = 1/3-q$, which only makes sense for $q \le 1/3$. For the final case, suppose $p=2/3$. Then type-$1$ agents do not enter since $1 < 3p$; that is, $m\{(1,0)\} = 1-q$ and  $m\{(1,1)\}=0$. This implies
\[
2/3 = p = m\{(1,1)\} + m\{(2,1)\} = m\{(2,1)\}.
\]
Since $q = m\{(2,0)\} + m\{(2,1)\} = m\{(2,0)\}  + 2/3$, we must have
$q \ge 2/3$. 
\end{proof}

Similarly, in the $n$-player game, we can argue that there exists a Nash equilibrium with type vector $\vec{w}$, for every fixed $\vec{w}=(w_1,\ldots,w_n) \in \W^n$. To construct an example, there are three cases, depending again on the fraction $q$ of $(w_1,\ldots,w_n)$ which equal $2$. In each case, we construct one example (though there may be more) of an equilibrium, recalling that agent $i$ \emph{enters the market} if $x_i=1$:
\begin{enumerate}
\item Suppose $q \in [1/3,2/3]$. All type-$2$ agents enter, while none of the type-$1$ agents enter.
\item Suppose $q < 1/3$. All type-$2$ agents enter. Let $k$ be the greatest integer less than or equal to $n(1/3 - q)$. Then $k$ of the type-$1$ agents enter, and the rest do not.
\item Suppose $q > 2/3$. All type-$1$ agents choose not to enter. Let $k$ be the greatest integer less than or equal to $2/3$. Then $k$ of the type-$2$ agents enter, and the rest do not.
\end{enumerate}
Note that we have constructed multiple equilibria in the latter cases, although they share a common type-action distribution.

Now that we have computed the Cournot-Nash equilibria and are confident that $n$-player equilibria exist, we are ready to apply Theorem \ref{th:intro-conditional-limit}.
Now, let $\mu_n$ denote the empirical type-action distribution as usual, where the types are i.i.d.\ samples from the distribution $\lambda_{2/3}$. That is, each of the $n$ agents is independently assigned type $2$ with probability $2/3$ and type $1$ with probability $1/3$. By Theorem \ref{th:limit} and Proposition \ref{pr:entrygame}, we know that $\mu_n$ converges a.s. to the unique element $m_{2/3}$ of $\M(\lambda_{2/3})$.
Let us show using Theorem \ref{th:intro-conditional-limit} that, for $r \in (1/3,2/3)$,
if we condition on the rare event $\{\mu_n^x(1) \le r\}$, then $\mu_n
\rightarrow m_r$. More precisely, the conditional law of $\mu_n$
converges to the point mass at $m_r$. Intuitively, this rare event
means that no more than a fraction of $r$ of the agents enters the
market, and the most likely way for this to happen (asymptotically) is
for precisely a fraction of $r$ of the agents to enter.

Let $1/3 < r < 2/3$, and consider the set 
\[
A = \left\{m \in \P(\W\times\X) : m^x\{1\} \le r\right\}.
\]
We then compute
\begin{align*}
\inf_{m \in A \cap \M}H(m^w|\lambda_{2/3}) &= \inf\left\{H(m_q^w|\lambda_{2/3}) : q\in [0,1], \ m_q \in A\right\} \\
	&= \inf\left\{H(\lambda_q|\lambda_{2/3}) : q \in [0,r]\right\} \\
	&= \inf\left\{H(\lambda_q|\lambda_{2/3}) : q \in [0,r)\right\} \\
	&= \inf_{m \in A^\circ \cap \M}H(m^w|\lambda_{2/3}),
\end{align*}
where the second to last equality holds by continuity of $q \mapsto
H(\lambda_q|\lambda_{2/3})$ at $q=r$. Moreover, the unique minimizer
on the left-hand side is $m_r$,  since $q \mapsto H(\lambda_q|\lambda_{2/3})$ is strictly decreasing for $0 < q < 2/3$. This shows that the assumption \eqref{def:conditional-assumption-original} of Theorem \ref{th:intro-conditional-limit} holds, and also that the set $S(A)$ therein is simply the singleton $\{m_r\}$.

\begin{remark}
Interestingly, a simple variant of the above game yields a
tractable example in which the Cournot-Nash
equilibria are not unique and yet Theorem \ref{th:intro-conditional-limit} can be applied. For instance, suppose $\W=\{-1,1\}$ and $\X=\{0,1\}$, with
\[
F(m,w,x) = -x(2m^x\{1\} + w).
\]
Note that there is no minus sign in front of $2m^x\{1\}$, so it is not
really an entry game; agents are now \emph{encouraged} to participate
(i.e., choose $x=1$) when other agents participate. Setting
$m^1_q(dw,dx)=\lambda_q(dw)\delta_1(dx)$, it can be checked that
$m^1_q \in \M(\lambda_q)$ for each $q \in [0,1]$; that is, it is
always an equilibrium for every agent to participate. However, there
are two (resp. three) equilibria for $q=1/2$ (resp. $q <
1/2$). Nonetheless, if $\mu_n$ is the empirical type-action
distribution when types are sampled from $\lambda_p$, where $p > 1/2$,
we can find a limit theorem for the law of $\mu_n$ conditioned on the
event $\{\mu_n \in A\}$ where $A = \{m : m\{(-1,1)\} \le 1-r\}$, for
$r \in (p,1)$ close to $p$. Indeed, we can check that the assumption
\eqref{def:conditional-assumption} holds, and the unique element of
$S(A)$ is $m_r$.  There is even a critical value of $r$ for which
\eqref{def:conditional-assumption} holds, but $S(A)$ is no longer a
singleton.  We omit the details of these calculations, with the
remark mainly serving to illustrate the need for the generality of
Theorem \ref{th:intro-conditional-limit}.
\end{remark}

\subsection{Proof of probabilistic bounds on the price of anarchy} 
\label{se:PoA-proofs}

To prove Proposition \ref{pr:PoA}, we rework the notation of Section \ref{se:PoA} as we did in Section \ref{subs-coupling}. Recall the notation $\mathrm{Gr}(\C) = \{(w,x) : x \in \C(w)\}$. For $(\lambda,u)$ belonging to the domain $\DG$ defined in \eqref{def-dg}, define $\A(\lambda,u) \subset \P(\W\times\X)$ by
\[
\A(\lambda,u) := \left\{m \in \mathcal{E}_u(\W\times\X) : m(\mathrm{Gr}(\C))=1, \ m^w = \lambda\right\},
\]
Interpret $\A(\lambda,u)$ as the set of admissible type-action distributions. Recall the notation
\[
V(m) = \int_{\W\times\X} F(m,w,x)m(dw,dx),
\]
for $m \in \P(\W\times\X)$. The price of anarchy is now defined as the function $\mathfrak{P} : \DG \rightarrow [1,\infty]$ given by
\[
\mathfrak{P}(\lambda,u) = \left. \sup_{m \in \NN(\lambda,0,u)}V(m) \right\slash \inf_{m \in \A(\lambda,u)}V(m).
\]

\begin{lemma} \label{le:PoAusc}
Suppose $V> 0$ pointwise. Then $\mathfrak{P}$ is upper semicontinuous on $\DG$.
\end{lemma}
\begin{proof}
First, $(\lambda,u) \mapsto \sup_{m \in \NN(\lambda,0,u)}V(m)$ is
upper semicontinuous because $V$ is continuous and because, by Proposition \ref{pr:Nash-UHC}, $\NN$ is upper hemicontinuous and has compact values (see \cite[Lemma 17.30]{aliprantisborder}). It suffices (since $V > 0$) to show that the denominator $\inf_{m \in \A(\lambda,u)}V(m)$ is lower semicontinuous and strictly positive. For both of these claims it suffices to show that the set-valued map $\A$ is upper hemicontinuous and has compact values (again by \cite[Lemma 17.30]{aliprantisborder}).
To prove this, we again use the sequential characterization of upper
hemicontinuity.
Fix a convergent sequence $(\lambda_n,u_n) \rightarrow
(\lambda,u)$ in $\DG$, and let $m_n \in \A(\lambda_n,u_n)$ for each
$n$. We must show that there exist $m \in \A(\lambda,u)$ and a
subsequence $\{m_{n_k}\}$ which converges to $m$. Because $m_n^w =
\lambda_n$, the sequence $\{m_n^w\} \subset \P(\W)$ is tight. Because $\X$ is compact, the sequence $\{m_n\} \subset \P(\W\times\X)$ is tight and thus precompact by Prokhorov's theorem. Let $m$ denote any limit point, and abuse notation by assuming $m_n \rightarrow m$.

It remains to show that $m$ belongs to $\A(\lambda,u)$.
As $\mathrm{Gr}(\C)$ is closed, the Portmanteau theorem implies $m(\mathrm{Gr}(\C)) = \lim_n m_n(\mathrm{Gr}(\C))=1$. Clearly 
\[
m^w = \lim_n m^w_n = \lim_n \lambda_n = \lambda, 
\]
where the limits are  in distribution.
Finally, to check that $m \in \mathcal{E}_{u}(\W\times\X)$, there are two cases. If $u=0$, then $\mathcal{E}_{u}(\W\times\X)=\P(\W\times\X)$ and there is nothing to prove. Otherwise, $u_n = u$ for all sufficiently large $n$, which implies $m_n$ and thus $m$ belong to the closed set $\mathcal{E}_{u}(\W\times\X)$.
\end{proof}

\begin{proof}[Proof of Proposition \ref{pr:PoA}]
The notation of Proposition \ref{pr:PoA} translates as follows to the present notation, for $n \ge 1$ and $w_1,\ldots,w_n \in \W$:
\[
\mathrm{PoA}_n(w_1,\ldots,w_n) = \mathfrak{P}\left(\frac{1}{n}\sum_{i=1}^n\delta_{w_i},\frac{1}{n}\right), \quad\quad \text{ and } \quad\quad \mathrm{PoA}(\lambda) = \mathfrak{P}(\lambda,0).
\]
When $(W_i)_{i=1}^\infty$ are i.i.d.\ with distribution $\lambda_0$, we know that $\frac{1}{n}\sum_{i=1}^n\delta_{W_i} \rightarrow \lambda_0$ almost surely. By Lemma \ref{le:PoAusc}, 
\begin{align*}
\limsup_{n\rightarrow\infty}\mathrm{PoA}_n(W_1,\ldots,W_n) \le
\mathrm{PoA}(\lambda_0), \ a.s.
\end{align*}
To prove the second claim, we apply Sanov's theorem. Consider the set
\begin{align*}
B = \left\{(\lambda,u) \in \DG: \mathfrak{P}(\lambda,u) \ge r\right\}.
\end{align*}
Because $\mathfrak{P}$ is upper semicontinuous on $\DG$, the set $B$
is closed in $\DG$ and thus in $\P(\W) \times [0,1]$. By Sanov's
theorem,
$\left(\frac{1}{n}\sum_{i=1}^n\delta_{W_i},\frac{1}{n}\right)$
satisfies an LDP 
 on $\P(\W) \times [0,1]$ with good rate function
\[
J(\lambda,u) = \begin{cases}
H(\lambda|\lambda_0) &\text{if } \lambda \in \P(\W), \ u =0, \\
\infty &\text{otherwise}.
\end{cases}
\]
Thus, we have 
\begin{align*}
\limsup_{n\rightarrow\infty}\frac{1}{n}\log\PP(\mathrm{PoA}_n(W_1,\ldots,W_n) \ge r) 	
	&= \limsup_{n\rightarrow\infty}\frac{1}{n}\log\PP\left(\left(\frac{1}{n}\sum_{i=1}^n\delta_{W_i},\frac{1}{n}\right) \in B\right) \\
	&\le -\inf_{(\lambda,u) \in B}J(\lambda,u) \\
	&= -\inf\left\{H(\lambda|\lambda_0) : \lambda \in \P(\W), \ \mathrm{PoA}(\lambda) \ge r\right\}.
\end{align*}
To prove the lower bound, simply apply the lower bound of Sanov's theorem to the set $B^c$.
\end{proof}

\appendix

\section{The upper Vietoris topology} \label{ap:vietoris}

For this section, fix a Hausdorff topological space $\Y$, and let $\CY$ denote
the set of closed subsets of $\Y$. The \emph{upper Vietoris topology}
on $\CY$ is the one generated by the base $\{E^+ : E \subset \Y \text{
  is open}\}$, where we define $E^+ := \{A \in \CY: A \subset E\}$ for sets $E \subset \Y$. This section collects a few basic facts about this topology. First, notice that $\CY$ is not Hausdorff, because if $A_1$ and $A_2$ are two distinct closed subsets of $\Y$ with $A_1 \cap A_2 \neq \emptyset$, then $A_1$ and $A_2$ cannot be separated by open sets; indeed, if $E_i \subset \Y$ is open with $A_i \in E_i^+$ for $i=1,2$, then $A_i \subset E_i$, and $A_1 \cap A_2 \subset E_1 \cap E_2$ implies that $E_1^+$ and $E_2^+$ are not disjoint.

Recall from Section \ref{se:proofs} the definition of upper
hemicontinuity of a set-valued function (with closed values) between
two topological spaces.  Namely, if $\Z$ is another topological space and
$\Gamma : \Z \rightarrow \CY$, then $\Gamma$ is upper hemicontinuous if and only if $\{z \in \Z : \Gamma(z) \subset E\}$ is open for every open set $E \subset \Y$. Equivalently, $\Gamma$ is upper hemicontinuous if and only if $\Gamma^{-1}(E^+) := \{z \in \Z : \Gamma(z) \in E^+\}$ is open for every open set $E \subset \Y$. Because $\{E^+ : E \subset \Y \text{ is open}\}$ is a base for the upper Vietoris topology, this immediately proves the following observation:

\begin{lemma} \label{le:uhcVietoris}
Let $\Z$ be another topological space.
A mapping $\Gamma : \Z \rightarrow \CY$ is continuous with respect to the upper Vietoris topology if and only if it is upper hemicontinuous as a set-valued map.
\end{lemma}

\begin{lemma} \label{le:lhcVietoris}
Let $g : \Y \rightarrow \R$, and define $G : \CY \rightarrow \R$ by $G(A) = \sup_{y \in A}g(y)$. If $g$ is upper semicontinuous, then so is $G$. If $g$ is continuous, and if $A \in \CY$ is a compact set such that $g$ is constant on $A$ (i.e., $g(y)=g(y')$ for all $y,y' \in A$), then $G$ is continuous at $A$.
\end{lemma}
\begin{proof}
The identity map on $\CY$ is continuous and thus, by Lemma \ref{le:uhcVietoris}, can be seen as an upper hemicontinuous set-valued map. The first claim then follows from \cite[Lemma 17.30]{aliprantisborder}. To prove the second claim, define a set-valued map $\Gamma : \CY \rightarrow 2^\R$ by $\Gamma(B) = \{g(y) : y \in B\}$. Then we may write $G(B) = \sup_{r \in \Gamma(B)}r$ for every $B \in \CY$. Because $g$ is continuous, $\Gamma$ is upper hemicontinuous \cite[Theorem 17.23]{aliprantisborder}. Note $\Gamma(A) = \{r_0\}$ is a singleton, by assumption. 

To prove $G$ is lower semicontinuous at $A$, suppose $A_\alpha$ is a net in $\CY$ converging to $A$. Choose arbitrarily $r_\alpha \in \Gamma(A_\alpha)$ for each $\alpha$. By upper hemicontinuity of $\Gamma$ and compactness of $\Gamma(A)$, the net $(r_\alpha)$ has a limit point in $\Gamma(A)$ by \cite[Theorem 17.16]{aliprantisborder}. Hence, $r_\alpha \rightarrow r_0$. Thus
\[
\liminf_\alpha G(A_\alpha) \ge \liminf_\alpha r_\alpha = r_0 = G(A).
\]
\end{proof}

Next, we identify the interiors and closures of certain subsets of $\CY$. Recall that the interior of a set is simply the union of its open subsets, and the closure of a set is the intersection of all closed sets containing it. Take note also that $\{E^+ : E \subset \Y \text{ is open}\}$ is a base, and so every open set in $\CY$ can be written as a union of these base elements. In the following lemma, especially the proof, we will be applying repeated complements, interiors, and $+$ operations, and we prefer to keep parentheses to a minimum by writing, e.g., $E^{c\,\circ \,+}$ in place of $((E^c)^\circ)^+$.

\begin{lemma} \label{le:Vietoris-setoperations}
Let $E \subset 
\Y$ be any set, and define
\begin{align*}
\mathfrak{U} &:= \{A \in \CY : A \cap E \neq \emptyset\} = E^{c\,+\,c}, \text{ and } \\
\mathfrak{W} &:= \{A \in \CY : A \subset E \} = E^+.
\end{align*}
The following hold:
\begin{enumerate}[(i)]
\item $\mathfrak{U}^\circ = \mathfrak{W}^\circ = \{A \in \CY : A \subset E^\circ\}$. In other words, $E^{c\,+\,c\,\circ} = E^{+\,\circ} = E^{\circ\,+}$.
\item $\overline{\mathfrak{U}} = \overline{\mathfrak{W}} = \{A \in \CY : A \cap \overline{E} \neq \emptyset\}$. In other words, $\overline{E^{c\,+\,c}} = \overline{E^+} = \overline{E}^{c\,+\,c}$.
\end{enumerate} 
\end{lemma}
\begin{proof} {\ } \\
\begin{enumerate}[(i)]
\item Suppose $\widetilde{E} \subset \Y$ is open. Then $\widetilde{E}^+ \subset \mathfrak{U}$ if and only if $A \cap E \neq \emptyset$ for every closed set $A \subset \widetilde{E}$. By considering $A = \{x\}$ for $x \in \widetilde{E}$ (which is closed because $\Y$ is Hausdorff), we see that $\widetilde{E}^+ \subset \mathfrak{U}$ if and only if $\widetilde{E} \subset E$. Thus the interior of $\mathfrak{U}$ is the union over all sets $\widetilde{E}^+$ such that $\widetilde{E}$ is open and $\widetilde{E} \subset E$, and the largest such set is given by $\widetilde{E} = E^\circ$. This shows $\mathfrak{U}^\circ = E^{\circ\,+} = \{A \in \CY : A \subset E^\circ\}$. To show that $\mathfrak{W}^\circ = E^{\circ\,+}$, note first that clearly $E^{\circ\,+} \subset \mathfrak{W}^\circ$. On the other hand, if $\widetilde{E} \subset \Y$ is any open set such that $\widetilde{E}^+ \subset \mathfrak{W}$, then $A \subset \widetilde{E}$ implies $A \subset E$ for every closed set $A \subset \Y$. Taking $A = \{x\}$, we conclude that $\widetilde{E}^+ \subset \mathfrak{W}$ implies $\widetilde{E} \subset E$, which in turn implies $\widetilde{E} \subset E^\circ$ and $\widetilde{E}^+ \subset E^{\circ\,+}$.
\item Recall the identities $\overline{A} = A^{c\,\circ\,c}$ and $\overline{A}^c = A^{c\,\circ}$, valid for any set $A$ in any topological space. Namely, apply (i) with $E^c$ in place of $E$ to get $E^{c\,+\,\circ} = E^{c\,\circ\,+}$, and thus
\begin{align*}
\overline{\mathfrak{U}} = \overline{E^{c\,+\,c}} = E^{c\,+\,c\,c\,\circ\,c} = E^{c\,+\,\circ\,c} = E^{c\,\circ\,+\,c} = \overline{E}^{c\,+\,c}.
\end{align*}
Similarly, apply (i) with $E^c$ in place of $E$ to get  $E^{+\,c\,\circ} = E^{c\,c\,+\,c\,\circ} = E^{c\,\circ\,+}$, and thus
\begin{align*}
\overline{\mathfrak{W}} = \overline{E^+} = E^{+\,c\,\circ\,c} = E^{c\,\circ\,+\,c} = \overline{E}^{c\,+\,c}.
\end{align*}
\end{enumerate}
\end{proof}

\section{Nonatomic Congestion games} \label{ap:congestiongames}
This section is devoted to existence and uniqueness results for nonatomic congestion games, namely the proofs of Propositions \ref{pr:congestion-potential} and \ref{pr:congestion-unique}.

\subsection*{Proof of Proposition \ref{pr:congestion-potential}}
For each $m,\widetilde{m} \in \P(\W)$, define the directional derivative
\begin{align*}
D_{\widetilde{m}}U(m) := \frac{d}{d\epsilon}U(m + \epsilon(\widetilde{m}-m))|_{\epsilon=0}.
\end{align*}
Noting that $\ell_e(m) = \int 1_{\{e \in x\}}\,m(dw,dx)$ for each $e \in E$, we compute
\begin{align*}
D_{\widetilde{m}}U(m) &= \sum_{e \in E}c_e(\ell_e(m))\frac{d}{d\epsilon}\ell_e(m + \epsilon(\widetilde{m}-m))|_{\epsilon=0} \\
	&= \sum_{e \in E}c_e(\ell_e(m))\int  1_{\{e \in x\}}(\widetilde{m}-m)(dw,dx) \\
	&= \int \sum_{e \in x}c_e(\ell_e(m))(\widetilde{m}-m)(dw,dx) \\
	&= \int F(m,x)(\widetilde{m}-m)(dw,dx).
\end{align*}
By definition, $m \in \M(\lambda)$ if and only if $\int F(m,x)(\widetilde{m}-m)(dw,dx) \ge 0$ for every $\widetilde{m} \in \A(\lambda)$. On the other hand, $m$ minimizes $U(\cdot)$ on $\A(\lambda)$ if and only if $D_{\widetilde{m}}U(m) \ge 0$ for every $\widetilde{m} \in \A(\lambda)$
\hfill \qedsymbol

\subsection*{Proof of Proposition \ref{pr:congestion-unique}}
Denote a generic element of $\mathbb{T}$ by $K=(K_{i,j})_{i,j}$, where $i=1,\ldots,|\W|$ and $j=1,\ldots,|\X|$.
Define $U_\lambda : \mathbb{T} \rightarrow \R$ by $U_\lambda(K) = U(m_K)$, where $m_K \in \A(\lambda)$ is given by $m_K\{(w_i,x_j)\} = \lambda_iK_{i,j}$, where $\lambda_i := \lambda\{w_i\}$. Note that
\[
\ell_e(K) := \ell_e(m_K) = \sum_{i=1}^{|\W|}\sum_{j=1}^{|\X|}1_{\{e \in x_j\}}\lambda_iK_{i,j}.
\]
Then $\partial_{K_{i,j}}\ell_e(K) = \lambda_i1_{\{e \in x_j\}}$, and so
\begin{align*}
\partial_{K_{i,j}}U_\lambda(K) &= \partial_{K_{i,j}}\sum_{e \in E}\int_0^{\ell_e(K)}c_e(s)ds = \sum_{e \in E}c_e(\ell_e(K))\lambda_i1_{\{e \in x_j\}}, \\
\partial_{K_{i',j'}}\partial_{K_{i,j}}U_\lambda(K) &= \sum_{e \in E}c_e'(\ell_e(K))\lambda_i\lambda_{i'}1_{\{e \in x_j\}}1_{\{e \in x_{j'}\}}.
\end{align*}
Hence, for any $T=(T_{i,j}) \in \mathbb{T}$, we have
\begin{align*}
\sum_{i,j}\sum_{i',j'}T_{i,j}T_{i',j'}\partial_{K_{i',j'}}\partial_{K_{i,j}}U_\lambda(K) &= \sum_{e \in E}c_e'(\ell_e(K))\sum_{i,j}\sum_{i',j'}T_{i,j}T_{i',j'}\lambda_i\lambda_{i'}1_{\{e \in x_j\}}1_{\{e \in x_{j'}\}} \\
	&=  \sum_{e \in E}c_e'(\ell_e(K))\left(\sum_{i,j}T_{i,j}\lambda_i1_{\{e \in x_j\}}\right)^2.
\end{align*}
Note that $c_e' > 0$ for each $e \in E$ by assumption, and also the squared sum in the last expression is strictly positive for some $e$ since $T \neq 0$ belongs to the span of  $(\lambda_i1_{\{e \in x_j\}})_{i,j}$ by assumption. This shows that the Hessian of $U_\lambda$ is positive definite everywhere, and so $U_\lambda$ has a unique minimizer on the compact convex set $\mathbb{T}$.

\hfill\qedsymbol

\bibliographystyle{amsplain}
\bibliography{staticMFG}

\end{document}